\documentclass[11pt]{amsart}
\usepackage[utf8]{inputenc}

\usepackage{amsmath, amstext, amssymb, bold-extra, graphicx, tikz, amsthm, amsfonts, hyperref,longtable, bm, nicefrac,color,bbm,mathrsfs}
\usepackage[floatrow]{chemstyle}
\usetikzlibrary{calc}
 \usepackage[nocompress]{cite}

\usepackage{mdwlist}
\usepackage{todonotes}
\usepackage{enumerate}
\RequirePackage{color}
\definecolor{my-blue}{rgb}{0.0,0.0,0.6}
\definecolor{my-red}{rgb}{0.5,0.0,0.0}
\definecolor{my-green}{rgb}{0.0,0.5,0.0}
\definecolor{nicos-red}{rgb}{0.75,0.0,0.0}
\definecolor{light-gray}{gray}{0.6}
\definecolor{really-light-gray}{gray}{0.8}
\definecolor{sussexg}{rgb}{0.0,0.5,0.5}
\definecolor{sussexp}{rgb}{0.5,0.0,0.5}

\usepackage{tikz}
\usepackage{subcaption}
\usepackage{mdwlist}
\usepackage{todonotes}
\usepackage{enumerate}
\usepackage{stackengine}
\usetikzlibrary{datavisualization} %for graphs and pictures
\usetikzlibrary{datavisualization.formats.functions} %for graphs and pictures
\usetikzlibrary{shapes.misc}
\tikzset{cross/.style={cross out, draw=black, minimum size=2*(#1-\pgflinewidth), inner sep=0pt, outer sep=0pt},
cross/.default={2pt}}

\newtheorem{theorem}{\textcolor{my-red}{\sc Theorem}}[section]
\newtheorem{lemma}[theorem]{\textcolor{my-green}{\bf { Lemma}}}
\newtheorem{proposition}[theorem]{\textcolor{my-green}{{\bf Proposition}}}
\newtheorem{corollary}[theorem]{\textcolor{my-green}{{\bf Corollary}}}

\newtheorem{remark}[theorem]{\textcolor{blue}{\bf Remark}}
\numberwithin{equation}{section}
\theoremstyle{remark}

\newcommand{\be}{\begin{equation}}
\newcommand{\ee}{\end{equation}}

\providecommand{\P}[1]{\langle#1\rangle}
\newcommand{\fl}[1]{\left\lfloor{#1}\right\rfloor}

\def\bE{\mathbb{E}}
\def\bN{\mathbb{N}}
\def\bP{\mathbb{P}}
\def\bQ{\mathbb{Q}}
\def\bR{\mathbb{R}}
\def\bZ{\mathbb{Z}}

\def\e{\varepsilon}

\def\om{\omega}

 \def\Z{\bZ}
\def\Q{\bQ}
\def\R{\bR}
\def\N{\bN}
\def\E{\bE}
\def\P{\bP}

\newcommand\barbelow[1]{\stackunder[1.2pt]{$#1$}{\rule{.8ex}{.075ex}}}

\DeclareMathOperator*{\argmin}{arg\,min}

\calclayout

\allowdisplaybreaks

\begin{document}
\date{\today}
\title[LDP for Bernoulli last passage times] 
{A Large deviation principle for last passage times in an asymmetric Bernoulli potential}

\author[F.~Ciech]{Federico Ciech}
\address{Federico Ciech\\ University of Sussex\\ Department of  Mathematics \\ Falmer Campus\\ Brighton BN1 9QH\\ UK.}
\email{F.Ciech@sussex.ac.uk}
\urladdr{http://www.sussex.ac.uk/profiles/395447} 

\author[N.~Georgiou]{Nicos Georgiou}
\address{Nicos Georgiou\\ University of Sussex\\ Department of  Mathematics \\ Falmer Campus\\ Brighton BN1 9QH\\ UK.}
\email{N.Georgiou@sussex.ac.uk}
\urladdr{http://www.sussex.ac.uk/profiles/329373} 

\thanks{N. Georgiou was partially supported by the EPSRC First Grant EP/P021409/1: The flat edge in last passage percolation.}

% \date{\today}
\begin{abstract}
	We prove a large deviation principle  and give an expression for the rate function, for the last passage time in a Bernoulli environment. 
	The model is exactly solvable and 
	its invariant version satisfies a Burke-type property. 
	Finally, we compute explicit limiting logarithmic moment generating functions 
	for both the classical and the invariant models. The shape function of this model exhibits a flat edge in certain directions, and
	 we also discuss the rate function and limiting  log-moment generating functions in those directions. 
\end{abstract}

\maketitle 
\setcounter{tocdepth}{1}
\tableofcontents

\section{Introduction}

\subsection{Brief description of the model} We study large deviations for the last passage time in a Bernoulli environment. The original model was introduced in \cite{Sep-98-aop-2} as a simplified model of directed first passage percolation. In this model, the environment $\eta = \{ \eta^{\kappa, \lambda}_v \}_{v \in \Z^2_+}$ is a collection of i.i.d.\ Bernoulli($p$)  under a background measure $\P$ with marginals 
\[
\P\{ \eta_v = \lambda\} = p = 1 - \P\{ \eta_v = \kappa\}, \quad \kappa > \lambda \in \R_+, v \in \Z^2_+.
\] 
The set of admissible paths from $(0,0)$ to $(m,n) \in \Z^2_+$ is denoted by $\Pi_{m,n}$ and it contains all paths of the form  
\[
\pi_{(0,0),(m,n)}=\{0=v_0,v_1,\dots,v_{m+n}=(m,n)\}, 
\]
so that $v_{i+1} - {v_i} \in \mathcal R = \{ e_1, e_2\}$. We say that $\mathcal R$ is the set of admissible steps. The random variable under consideration is the ``first passage time" 
\be \notag
L^{\kappa, \lambda, p}_{(0,0), (m,n)} = \inf_{ \pi \in \Pi_{m,n}} \sum_{v_i \in \pi}V (T_{v_i}\eta, v_{i+1}-v_i), 
\ee
where $T_v$ denotes the shift by $v \in \Z^2_+$ and $V: \Omega \times \mathcal R \to \R$ is the potential function given by 
\be\notag
V(\eta, z) = \eta_{e_1}1\!\!1\{ z = e_1\}+ \tau_01\!\!1\{ z = e_2\}.
\ee
Value $\tau_0$ was constant and fixed from the beginning. The interest was to find the explicit shape function 
\[
\mu(s,t) = \lim_{n \to \infty} \frac{L^{\kappa, \lambda, p}_{(0,0), (\fl{ns},\fl{nt})}}{n}.
\]
The model can be mapped into a last passage directed percolation by two observations. First, because the admissible paths are directed the number of vertical increments $z =e_2 \in \mathcal R$ are fixed for any fixed endpoint $(m.n)$ (in fact they are $n$) and the cost for crossing them is deterministic $\tau_0$. Thus, for simplification $\tau_0$ can be set to be zero. Second,  since $\lambda < \kappa$, to minimize $L^{\kappa, \lambda, p}$ one should try and take horizontal steps $e_2 \in \mathcal R$ when the value of the environment at the target site is $\lambda$. Define new environment 
\be\label{eq:envom}
\omega_v = \frac{1}{\kappa-\lambda}(\kappa - \eta_v) \sim {\rm Ber}(p) \in \{ 0 , 1\}.
\ee
Then define the last passage time
\be\label{eq:lptV}
 G^V_{(0,0),(m,n)} = \max_{\pi_{(0,0),(m,n)} \in \Pi_{(0,0), (m,n)}} \bigg\{ \sum_{v_i \in \pi} V(T_{v_i}\omega, v_{i+1}-v_i) \bigg\}.
\ee
The value of  $G^V$ gives the number of horizontal steps through environment $\omega_v = 1$, equivalently $\eta_v = \lambda$. 
Each of the remaining horizontal steps contributes  $\kappa$ to $L^{\kappa, \lambda, p}$ and therefore we have
\be\label{conn}
L^V_{(0,0),(m,n)} =(\lambda-\kappa) G^{V}_{(0,0),(m,n)} + \kappa m+\tau_0 n.
\ee
Therefore, for simplicity we study the last passage time $G^V$ given by \eqref{eq:lptV}, in environment $\omega$ given \eqref{eq:envom}, under potential $V$ given by  
\be\label{eq:potential}
V(\omega, z) = \omega_{e_1}1\!\!1\{ z = e_1\}.
\ee
By \eqref{conn} one can translate all results to $L^V$. Now that $V$ is specified we omit it from the notation. We also omit $(0,0)$ as the starting point, when it is implied. Therefore, the last passage time \eqref{eq:lptV} is simply denoted by $G_{m,n}$. If the starting point is $(k, \ell)$ we write $G_{(k, \ell), (m,n)}$.

\begin{figure}
\centering
		\begin{tikzpicture}[>=latex,scale=0.8]
			\draw[<->](6.5,0)--(0,0)--(0, 6.5); 
			%maximal path		
			\draw[line width=3pt, color=sussexg](0,0)--(0,1)--(2,1)--(2,2)--(3,2)--(3,3)--(4,3)--(4,5)--(5,5)--(6,5)--(6,6)--(6.3,6);
					
			\foreach \x in {1,2,...,6}
							{
							\draw(\x,0)--(\x,6.3);
							\draw(0,\x)--(6.3, \x);
							\fill[color=white](\x,0)circle(1.3mm);
							\draw(\x,0)circle(1.3mm);
							\fill[color=white](0,\x)circle(1.3mm);
							\draw(0,\x)circle(1.3mm);
							}				
			\foreach \x in {1,2,...,6}
						{
						\foreach \y in {1,2,...,6}	
							{
							\fill[color=white](\x,\y)circle(1.3mm);
							\draw(\x,\y)circle(1.3mm);
							}
						}
			\foreach \x in {1, 4}
						{
						\foreach \y in {1,3,4,6}	
							{
							\fill[color=nicos-red](\x,\y)circle(1.3mm);
							\draw(\x,\y)circle(1.3mm);
%							\draw(\x,-0.5)node{$A$};
%							\draw(-0.5,\y)node{$A$};
							}
						}
			\foreach \x in {2}
						{
						\foreach \y in {1,2,3,4,5,6}	
							{
							\fill[color=nicos-red](\x,\y)circle(1.3mm);
							\draw(\x,\y)circle(1.3mm);
%							\draw(\x,-0.5)node{$A$};
%							\draw(-0.5,\y)node{$A$};
							}
						}
			\foreach \x in {3,5}
						{
						\foreach \y in {2,5}	
							{
							\fill[color=nicos-red](\x,\y)circle(1.3mm);
							\draw(\x,\y)circle(1.3mm);
%							\draw(\x,-0.5)node{$B$};
%							\draw(-0.5,\y)node{$B$};
							}
						}	
			\draw[line width=2pt, color=nicos-red](1,1)circle(4mm);	
			\draw[line width=2pt, color=nicos-red](2,1)circle(4mm);
			\draw[line width=2pt, color=nicos-red](3,2)circle(4mm);
			\draw[line width=2pt, color=nicos-red](4,3)circle(4mm);	
			\draw[line width=2pt, color=nicos-red](5,5)circle(4mm);											
		\end{tikzpicture}
\caption{ A possible representation of the maximal path (green thick line) given a fixed random environment. The red dots represent the sites where  $\om=1$ while the white dots are when $\om=0$. The circled dots are the ones that the maximal path collects following the potential \eqref{eq:potential}. Finally note that despite there is a column of red dots the maximal path doesn't spend a lot of time there. This is due to the fact that by \eqref{eq:potential} it cannot collect through an $e_1$ step. This means that the maximal path does not really see columns with high density of Bernoulli successes.}
		\end{figure}
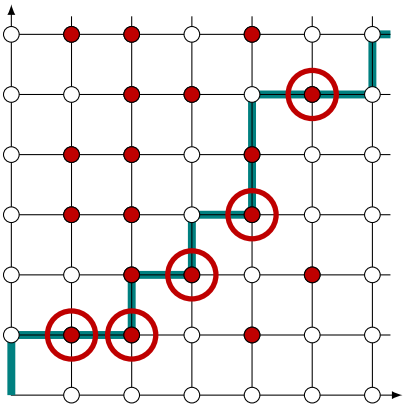
 
The law of large numbers for $G_{m,n}$ was first found in \cite{Sep-98-aop-2} by first obtaining invariant distributions for an embedded totally asymmetric particle system. 
Most recently the LLN was reproved in \cite{basdevant2015discrete} using an invariant boundary model with sources and sinks. The same idea was utilised in the same article for the discrete version of Hammersley's process \cite{hammersley1972few}, introduced in \cite{seppalainen1997increasing}. The theorem states
\begin{theorem}[The shape function for $G_{\fl{Ns}, \fl{Nt}}$ \cite{Sep-98-aop-2,basdevant2015discrete}]
\label{thm:LLNp}
Fix $p$ in $(0,1)$ and $(s,t) \in \R^2_+$. Then we have the explicit law of large numbers limit
\begin{align}
g_{pp}(s,t)=\lim_{N \to \infty} \frac{G_{\fl{Ns}, \fl{Nt}}}{N} %&= \inf_{p < u \le 1}\{ s \E(I^{(u)}) + t \E(J^{(u)})\} \notag\\
&=\begin{cases}\big(\sqrt{ps}+\sqrt{(1-p)t}\big)^2-t, \qquad &t<s\frac{1-p}{p}\\
s, \qquad& t\geq s\frac{1-p}{p}.
\end{cases}\label{eq:busopt}
\end{align}
\end{theorem}
This is a concave, symmetric, 1-homogeneous differentiable function which is continuous up to the boundaries of $\R^2_+$. Together with the shape function for the discrete Hammersley \cite{seppalainen1997increasing}, are the first completely explicit shape functions for which strict concavity is not valid. In fact, the formula above indicates one flat edge, for $t>s\frac{1-p}{p}s$.  

This simplified Bernoulli model was studied further in \cite{Grav-Tra-Wid-01} where Tracy-Widom distributional limits were obtained for this and a generalised inhomogeneous version where the probability of success of the Bernoulli environment changes with the first coordinate of the site. Then the LLN was used for certain estimates in proving generalised properties of the shape functions of last passage percolation in \cite{Mar-04}. 

Other models for which a flat edge of the shape function exists are common, and well studied.  A flat edge for the contact process was observed in \cite{durrett1981shape, durrett1983supercritical}. The discrete Hammersley model treated in \cite{seppalainen1997increasing, Cie-Geo-17, georgiou2010soft, georgiou2018optimality}  and the inhomogeneous model in \cite{Emr-16} allow for an exact derivation of the limiting shape function and they also exhibit two flat edges. Large deviations for the latter were obtained in \cite{emrah2015large}. In the present article, we also study the behaviour of large deviations in directions for which the shape is flat for this classical Bernoulli model. 

Historically, models of last passage percolation for which explicit invariant distributions could be obtained for the embedded particle system \cite{Ros-81,  aldous1995hammersley, seppalainen1997increasing, Sep-98-aop-2}, satisifed a Burke-type property that led to invariant boundary corner growth models 
\cite{Bal-Cat-Sep-06, Gro-01, Cat-Gro-05, Cat-Gro-06, Jo-00, basdevant2015discrete, Cie-Geo-17}. The same is true for this model, and we discuss the Burke-type property in  Section \ref{sec:model3}. Boundaries can be constructed using sequences of i.i.d.\ Bernoulli and Geometric random variables, that play the role of Busemann functions for the last passage time model. We use the invariance  granted by the boundaries to reprove Theorem \ref{thm:LLNp} and verify the variational formula of \cite{georgiou2016variational}. Because of the asymmetry of the model and the dependence of $V$ on the path increment $z \in \mathcal R$, the general findings of \cite{Geo-Ras-Sep-15a-, Geo-Ras-Sep-15b-} cannot be directly applied but the fact that invariant boundary models exist is a good indication that those theorem can hold more generally.

The Burke property guarantees enough analytical tractability to classify this as an exactly solvable model of the KPZ class \cite{Cor-12}. 
Several well-studied models of last passage percolation and directed polymers  exhibit this characteristic. 
Aside from the ones already mentioned above, there are also the continuum directed polymer studied in \cite{Am-11}, the log-gamma polymer introduced in  \cite{Sep-12-corr}, the polymer in a Brownian environment with continuous-time random walk paths, discovered in  \cite{OCo-Yor-01}, subsequently worked on by \cite{Mor-OCo-07, OCo-12,Sep-Val-10}, the strict-weak gamma polymer studied in \cite{corwin2015strict} and \cite{O-15} and the random walk in Beta-distributed random potential \cite{Bar-Cor-15-}. The exactly solvable planar polymer models with two admissible steps were recently classified in \cite{Ch-No-18}. Exactly solvable models which present environment inhomogeneity are for the corner growth model  \cite{Emr-16} and for totally asymmetric particle systems associated to growth models \cite{Bo-Pe-18, Kn-Pe-Sa-18}.

\subsection{Large deviations}
Large deviations rate functions for last passage times (for LPP)  and partition functions (for directed polymers) have been computed in several cases when the model is exactly solvable. Below $G$ stands for a generic last passage time random variable. Define the upper (or right) and lower (or left) tale for the rate function as 
\begin{align*}
\lim_{n\to\infty}-n^{-1}\log\P\{G_{n,n}\geq rn\}=J_u(r), \quad \lim_{n\to\infty}-n^{-2}\log\P\{G_{n,n}\leq rn\}=J_\ell(r),
\end{align*}
A priori the existence of the limits is not even guaranteed, and it depends for example on the potential $V$  and the environment $\om$ among other things.
The existence of $J_u(r)$ and $J_\ell(r)$ was proved for the exponential and geometric corner growth model in $\Z^2$ \cite{Jo-00}. 
An earlier work where the right-tail rate function is explicitly computed appeared in \cite{Sep-98-mprf-2}. Existence of the rate functions was also known in the case of of the Hammersley process.  Its fluctuations in the large deviations regime were studied in \cite{De-99}, obtaining also precise results for the upper and lower exponential tails.  An explicit right-tail rate function was computed in \cite{SeLarge-98}, using the invariant distributions for the particle system and studying deviations for the tagged particle. 
%A random matrix approach was used for the same in \cite{Ba-05}.
In the framework of particles systems, functional large deviation principle for Totally Asymmetric Simple Exclusion Process (TASEP), which is closely connected to Exponential LPP, was obtained, for the $n$-speed tail in \cite{Va-04,Je-phd-00} and for the $n^2$-speed tail recently in \cite{Ol-17}. 

Using the invariance structure offered by Burke's property, a right-tail large deviation rate function with speed $n$ for the partition function in the log-gamma polymer was proven \cite{Geo-Sep-13-}. Large deviations and KPZ fluctuations were computed for a random walk in a dynamic i.i.d.\ beta random environment in \cite{Ba-Se-18}. The idea of \cite{Geo-Sep-13-} was later extended for the free energy in the O'Connell-Yor polymer by \cite{Ja-15}, which is also a model with asymmetry like the one in this article. We utilise this techniques here, and we are also able to prove explicit limiting log-moment generating functions. 
 
The approach for the existence of the right tail rate function is probabilistic in nature and utilizes super-additivity and the explicit expression is computed using probabilistic arguments. In general, the speed $n^2$ and the existence of lower-tail rate functions remained elusive, including for non-solvable models of last passage percolation, if one was to use only probabilistic techniques. In \cite{Kes-86-aspects} it was shown under a boundedness condition on the environment that the $n^2$ speed was correct, but with no existence of the rate function results. This was for first passage percolation models (FPP). FPP and  LPP have the same qualitative behavior with the role of upper and lower tails reversed, an artefact of sub-additivity vs super-additivity. Existence of the $n^2$ speed rate function is proven in \cite{Ba-17} and the result is expected to extend for LPP with the same probabilistic approach. A variant of this result was earlier proved in \cite{Ch-03} for line-to-line first passage time.

\subsection{Structure of the paper}
The paper is organised as follows:
 In Section \ref{sec:model3} we discuss the Burke property and the invariant version of the model. 
 
 In Section \ref{sec:fullLDP} we prove a full large deviation principle (LDP) for $G_{\fl{Ns}, \fl{Nt}}$ at speed $n$. General properties of the rate function are also proven, including that its Legendre dual is the limiting logarithmic moment generating function (l.m.g.f.) of $G_{\fl{Ns}, \fl{Nt}}$ via Varadhan's lemma. Existence of the full LDP is a direct consequence of the existence of a right-tail rate function. We prove an explicit variational formula for the right-tail rate function and its Legendre dual, that we ten proceed to explicitly solve and obtain a closed formula in Section \ref{sec:ldpwb}. Finally, in Section \ref{sec:immgf} we prove an explicit expression of the limiting l.m.g.f.\ for the invariant boundary model. 
 
 The remaining part of the article are the Appendices, where we present variations of known results, that we needed tailored to our model. The last Appendix is a long computation needed in computing the explicit formulas. The goal was to make this article self-sufficient.

\subsection{Commonly used notation}
Throughout the paper, $\N$ denotes the natural numbers, and $\Z_+$ the non-negative integers. Symbol $G$ is always denoting a last passage time.  As we already mentioned, the superscript $V$ will be omitted as there is no confusion on the potential; in our case we always use \eqref{eq:potential}. Letter $\pi$ signifies a generic admissible path.

Bold-face letters (e.g.\ $\bf v$) indicate two-dimensional vectors (e.g.\ ${\bf w} = (w_1, w_2)$). In the rare cases where we write   ${\bf v} \le {\bf w}$ we mean the inequality holds coordinate-wise. 

The Legendre (convex) dual of a function $f :\R\to (-\infty,\infty]$ is defined $f^*(y) =\sup_{x\in\R}\{xy-f (x)\}$. The statement $f = f^{**}$ is used throughout the article without any special mention, and it is true if and only if $f$ is convex and lower semicontinuous, which is why we pay particular attention into having the rate function lower-semicontinuous at the boundaries of their set that they are finite. Finally, in two occasions we need the infimal convolution of two generalised convex functions $f, g$, and we write 
\[
f \Box g(r) = \inf_{x \in \R}\{ f(x) + g(r-x)\}.
\]
The important fact is that $(f\Box g)^* = f^*+g^*$.  We refer to \cite{Roc-70} for the necessary convex analysis.

\section{The invariant model}
\label{sec:model3} 

The boundary model is constructed defining different distribution of the weights on the two axes and they depends on a parameter $u \in (p,1]$. We can freely choose the value of this parameter according to its domain and different $u$ values defines different boundary distributions. The weight at the origin remains unchanged respect to the original model and it is set to $\om_0=0$. On the horizontal axis, for any $k \in \N$ we set the weights
$
\om_{ke_1} \sim \text{Bernoulli}( u ), 
$
with independent marginals 
\be \label{eq:xaxis}
\P\{ \om_{ke_1} = 1\} = u= 1 - \P\{ \om_{ke_1} = 0\}.
\ee
On the vertical axis, for any $k \in \N$, we set 
$\om_{ke_2} \sim \text{Geometric}\big( \frac{u-p}{u(1-p)} \big)$ with independent marginals 
\be\label{eq:yaxis} 
\P\{ \om_{ke_2} = \ell \} 
%= \frac{u-p}{u(1-p)}\Big(1 - \frac{u-p}{u(1-p)}\Big)^{\ell} 
=\frac{u-p}{u(1-p)}\Big(\frac{p(1-u)}{u(1-p)}\Big)^{\ell}  , \quad \ell \in \Z_+.
\ee
We do not alter the weights in the bulk $\{ \om_w\}_{w \in \N^2}$. Therefore they have i.i.d.\ Ber($p$) marginal distributions. 
We use the superscript $\om^{(u)}$ to highlight the fact that the distributions on the two axes are different than the ones in the bulk and they depend on $u$. To sum up, for any $i \ge 1, j \ge 1$, the $\om^{(u)}$ marginals are independent with marginals
\be \label{eq:omu}
\om^{(u)}_{i,j} \sim
\begin{cases}
\text{Ber}(p), &\text{ if } (i,j) \in \N^2,\\
\text{Ber}(u), &\text{ if }  i \in \N, j =0,\\
\text{Geom}\Big(\frac{u-p}{u(1-p)}\Big), &\text{ if }  i =0, j \in \N, \\
\delta_0, &\text{ if }  i =0, j =0.
\end{cases}
\ee
If we consider any path $\pi$ starting from $0$, we observe from the previous display that the invariant model allows for the possibility to $\pi$ to collect different weights along the two axes. In particular, if the path moves horizontally before entering the bulk,  then it collects the Bernoulli($u$) weights until it takes the first vertical step, and after that, it collects weight according to the potential function  \eqref{eq:potential}. If $\pi$ moves vertically from $0$ then it will collect the geometric weights on the vertical axis, and after it enters the bulk, it again collects according to $V$. This is the only difference from the potential $V$ of the i.i.d.\ model, namely while on the $y$-axis, the path can still collect positive weight. 

Given a parameter $u \in (p, 1]$ we define the last passage time for the invariant model from 0 to $w$ as
\begin{align}
G^{(u)}_{0, w} &= \max_{1\le k \le w\cdot e_1} \Big\{ \sum_{i=1}^k \om_{i e_1} + G_{ke_1+e_2, w} \Big\} \nonumber\\
&\phantom{xxxxxxxxxxxxx}\bigvee \max_{1\le k \le w\cdot e_2} \Big\{ \sum_{j=1}^k \om_{je_2}  + \om_{e_1+ke_2}+G_{e_1+ke_2, w} \Big\}. \label{eq:varform}
\end{align}

This formula comes from the variational equality using the above description. An explicit formula for the shape function for this model is given by the following theorem.

  \begin{theorem}\label{thm:LLNG}[Law of large numbers for $G^{(u)}_{\fl{Ns}, \fl{Nt}}$]
	For fixed parameter $p< u \le 1$ and $(s,t) \in \R^2_+$ we have 
	\be
	g_{pp}^{(u)}(s,t)=\lim_{N\to \infty} \frac{G^{(u)}_{\fl{Ns}, \fl{Nt}}}{N} = su + t\, \frac{p(1-u)}{ u - p } , \quad  \P -a.s.
	\ee 
\end{theorem} 

It is convenient to introduce to passage times, depending on the first step of the set of paths we are optimizing over. Define 
\be\label{lpthor}
G^{(u),\text{hor}}_{\fl{Ns},\fl{Nt}}=\max_{1\leq k\leq \fl{Ns}}\Big\{\sum_{i=1}^k \om_{i,0}+G_{(k,1),(\fl{Ns},\fl{Nt})}\Big\}
\ee
and 
\be\label{lptver}
G^{(u),\text{ver}}_{\fl{Ns},\fl{Nt}}=\max_{1\leq \ell\leq \fl{Nt}}\Big\{\sum_{j=1}^\ell \om_{0,j} + \om_{1,\ell}+G_{(1,\ell),(\fl{Ns},\fl{Nt})}\Big\}.
\ee
Then, by \eqref{eq:varform} 
\be\label{eq:gvg}
G^{(u)}_{\fl{Ns},\fl{Nt}}=G^{(u),\text{hor}}_{\fl{Ns},\fl{Nt}}\vee G^{(u),\text{ver}}_{\fl{Ns},\fl{Nt}}.
\ee
Passage times \eqref{lpthor} and \eqref{lptver} satisfy a law of large numbers as well, given in the next 
\begin{theorem} \label{thm:HOR}
Let $s,t\geq0$, $u\in(p,1]$.
\begin{enumerate}[(a)]
\item The following limit exists and is given by
\be \label{limgppu}
g_{pp}^{(u),\text{hor}}(s,t)= \lim_{N\to\infty}N^{-1}G^{(u),\text{hor}}_{\fl{Ns},\fl{Nt}}=
\begin{cases}
g_{pp}^{(u)}(s,t)\qquad&\text{if }t<s\frac{(u-p)^2}{p(1-p)},\\
g_{pp}(s,t)\qquad&\text{if }t\geq s\frac{(u-p)^2}{p(1-p)}.
\end{cases}
\ee
\item The following limit exists and is given by
\be \label{limgppv}
g_{pp}^{(u),\text{ver}}(s,t)= \lim_{N\to\infty}N^{-1}G^{(u),\text{ver}}_{\fl{Ns},\fl{Nt}}=
\begin{cases}
g_{pp}^{(u)}(s,t)\qquad&\text{if }t>s\frac{(u-p)^2}{p(1-p)},\\
g_{pp}(s,t)\qquad&\text{if }t\leq s\frac{(u-p)^2}{p(1-p)}.
\end{cases}
\ee
\end{enumerate}
\end{theorem}

As is usual in the exactly solvable models of last passage percolation, there is the notion of a \textit{characteristic direction}. In this case, for the model with boundaries for a given boundary parameter $u\in(p,1]$, there exists a unique direction $(m(N), n(N))$ whose scaled direction, as $N \to \infty$, converges to the macroscopic characteristic direction 
\be\label{eq:chardirmacro}
N^{-1}(m_{u}(N), n_{u}(N)) \to \Big( 1,\frac{(u-p)^2}{p(1-p)}\Big),
\ee
which gives that for large enough $N$ the endpoint $(m(N), n(N))$ is always below the critical line $y = \frac{1-p}{p}x$ that separates the flat edge from the strictly concave part of $g_{pp}(s,t)$ in Theorem \ref{thm:LLNp}. Here the characteristic direction already manifested in Theorem \ref{thm:HOR}  as the cutting line  between feeling the boundary effect  versus entering the bulk. 

We will prove Theorems \ref{thm:LLNG} and \ref{thm:HOR} at the end of this section.

To simplify the notation in what follows, set $w = (i,j) \in \Z_+^2$ and define the last passage time gradients by 
\be\label{eq:gradients}
I^{(u)}_{i+1,j} = G^{(u)}_{0, (i+1,j)}  - G^{(u)}_{0, (i,j)}, \quad \text{and} \quad J^{(u)}_{i,j+1} =  G^{(u)}_{0, (i,j+1)}  - G^{(u)}_{0, (i,j)}.
\ee
When there is no confusion we will drop the superscript $(u)$ from the $I, J$ notation. When $j = 0$ we have that $\{I^{(u)}_{i,0}\}_{i\in \N}$ are i.i.d.\ Bernoulli($u$) random variables since $I^{(u)}_{i,0} = \om_{i,0}$. When $i = 0$,   $\{J^{(u)}_{0,j}\}_{j \in \N}$ are i.i.d.\ Geometric$(\frac{u-p}{u(1-p)})$ random variables. There are three recursive equations satisfied in this model, which we summarize below.

\begin{lemma}
Let $u \in(p,1]$ and $(i, j) \in \N^2$. Then the last passage time can be recursively computed as 
\be\label{eq:LPPrec}
G^{(u)}_{0, (i, j)} = \max\big\{ G^{(u)}_{0, (i, j-1)}, \,\,G^{(u)}_{0, (i-1, j)} + \om_{i,j} \big\}.
\ee
Furthermore, the last passage time gradients satisfy the recursive equations 
\be\label{eq:4}
\begin{aligned}
I^{(u)}_{i,j}&=\max\{I^{(u)}_{i,j-1}-J^{(u)}_{i-1,j},\,\, \om_{i,j}\},\\
J^{(u)}_{i,j}&=(J^{(u)}_{i-1,j}-I^{(u)}_{i,j-1}+\om_{i,j})^+.
\end{aligned}
%\qquad \text{for }(i,j)\in\mathbb{N}^2.
\ee
\end{lemma}

\begin{proof}
Equation \eqref{eq:LPPrec} follows from the model description in the introduction. Note that the same recursive equation is actually satisfied for the model without boundaries. 
We indicatively prove  \eqref{eq:4} for the $J$ and the other one is obtained in a similar way. 
\begin{align*}
J^{(u)}_{i,j}&=G^{(u)}_{0,(i,j)}-G^{(u)}_{0, (i,j-1)}\\
&=\max\big\{G^{(u)}_{0, (i,j-1)}, G^{(u)}_{0,(i-1,j)}+\om_{i,j}\big\} - G^{(u)}_{0,(i,j-1)} \quad \text{ by \eqref{eq:LPPrec},}\\
&=\max\big\{G^{(u)}_{0, (i,j-1)}-G^{(u)}_{0, (i,j-1)},  G^{(u)}_{0,(i-1,j)}-G^{(u)}_{0, (i,j-1)}+\om_{i,j}\big\}\\
&=\max\big\{0,G^{(u)}_{0, (i-1,j)}-G^{(u)}_{0,(i,j-1)}+G^{(u)}_{0(i-1,j-1)}-G^{(u)}_{0(i-1,j-1)}+\om_{i,j}\big\}\\
&=(J^{(u)}_{i-1,j}-I^{(u)}_{i,j-1}+\om_{i,j})^+. \qedhere
\end{align*}
\end{proof}

Using the gradients \eqref{eq:4} and the environment $\{\om_{i,j}\}_{(i,j) \in \N^2}$ we also define new random variables $\alpha_{i,j}$ on $\Z_+^2$
\be
\label{eq:5}
\alpha_{i-1,j-1}=\min\{I^{(u)}_{i,j-1},J^{(u)}_{i-1,j}+\om_{i,j}\}\qquad \text{for }(i,j)\in\mathbb{N}^2.
\ee
Since the $I^{(u)}_{i,j}$ are Bernoulli, so are the $\alpha_{i,j}$. The following lemma gives the joint distribution of the vector
$(I^{(u)}_{i,j}, J^{(u)}_{i,j}, \alpha_{i-1,j-1})$. This is It is the analogue of Burke's property that is a common aspect of many solvable models.

\begin{lemma}[Burke's property]
\label{burke}
Let independent random variables be distributed by 
\be
(I^{(u)}_{i,j-1}, J^{(u)}_{i-1,j}, \om_{i,j}) \sim \Big( \textrm{Ber}(u), \text{Geom}\Big(\frac{u-p}{u(1-p)}\Big), \text{Ber}(p) \Big),
\ee
where we assume $u > p$.
Then, for $(i,j) \in \N^2$,  the vector obtained via equations \eqref{eq:4}, \eqref{eq:5} is a vector of mutually independent marginals, with the same distributions, i.e. 
\be
(I^{(u)}_{i,j}, J^{(u)}_{i,j}, \alpha_{i-1,j-1}) \sim \Big( \textrm{Ber}(u), \text{Geom}\Big(\frac{u-p}{u(1-p)}\Big), \text{Ber}(p) \Big).
\ee
\end{lemma}

We first present a series of key technical lemmas, and we encourage the reader 
familiar with these techniques to proceed to the proof of Proposition \ref{propJ}.

A down-right path $\psi$ on the lattice $\Z^2_+$ is an ordered  sequence of sites $\{ v_i \}_{i \in \Z}$ that satisfy 
\[
 v_i - v_{i-1} \in \{ e_1, - e_2 \}.
\] 
For a given down-right path $\psi$, define $\psi_i = v_i - v_{i-1}$ to be the $i$-th edge of the path and set  
\[
L_{\psi_i} = 
\begin{cases}
	I^{(u)}_{v_i}, & \text{if } \psi_i = e_1\\
	J^{(u)}_{v_{i-1}}, & \text{if } \psi_i = -e_2. 
\end{cases}
\]
Also define the interior sites  $\mathcal I _{\psi}$ of $\psi$ to be 
\[
\mathcal I _{\psi} = \{ w \in \Z^2_+: \exists\, v_i \in \psi \text{ s.t. } \,w < v_i \text{ coordinate-wise} \}. 
\] 
A convenient way to state Lemma \ref{burke} is the following. 
\begin{corollary} \label{cor:downrightpath}
	Fix a down-right path $\psi$. Then the random variables 
	\be\label{eq:dwrp}
	\{ \{\alpha_w\}_{w \in \mathcal I _{\psi}}, \{L_{\psi_i}\}_{i \in \Z} \}
	\ee
	are mutually independent, with marginals 
	\[
	\alpha_{w} \sim \text{Ber(p)}, \quad  L_{\psi_i} \sim 
\begin{cases}
	\text{Ber}(u), & \text{if } \psi_i = e_1\\
	\text{Geom}\Big( \frac{u - p}{ u(1-p) } \Big), & \text{if } \psi_i = -e_2. 
\end{cases}
	\]
\end{corollary} 

\begin{theorem}[Variational formula for the LLN of the non boundary model]
\label{thm:llnp}
Fix $p$ in $(0,1)$ and $(s,t) \in \R^2_+$. Then we have the explicit law of large numbers limit
\[
g_{pp}(s,t) = \inf_{p < u \le 1} \{ s\E(I^{(u)}) + t \E(J^{(u)})\} = \inf_{p < u \le 1} g_{pp}^{(u)}(s,t).   
\]

\end{theorem}

\begin{remark}
To see the characteristic direction manifesting in a different way, start from the formula in Theorem \ref{thm:llnp})
\[
g_{pp}(s,t) = \inf_{p < u \le 1} \{ s\E(I^{(u)}) + t \E(J^{(u)})\} = \inf_{p < u \le 1} g_{pp}^{(u)}(s,t).   
\]
This can be immediately seen from \eqref{eq:14} and the fact that $g_{pp}(s,t)  = sg_{pp}(1,ts^{-1})$ for example.
Without loss set $s = 1$. Then the $u^*$ that minimizes the expression above is $u^* = p + \sqrt{ t p(1-p)}$ if $t < q/p$ and 1 otherwise. Assume $t < q/p$. Solve the expression for $t$ we obtain
\[
t = \frac{(u^*-p)^2}{p(1-p)}.
\]
In other words, $g_{pp}(1,t) =  g_{pp}^{(u^*)}(1,t)$ and direction $(1,t)$ is characteristic according to \eqref{eq:chardirmacro} for the boundary model with parameter $u^*$. Note that the range of characteristic directions only covers the directions for which $g_{pp}^{(u)}(s,t)$ is strictly concave. The flat edge of $g_{pp}$ corresponds to $u^* = 1$.
\end{remark}

\begin{remark}
Along the characteristic direction the last passage time at point $N(m, n)$ it is expected to have variance of order $O(N^{2/3})$ for large $N$, while in the other directions the fluctuations of $G_{\fl{Ns},\fl{Nt}}$ to have order of magnitude $N^{1/2}$ and they are asymptotically Gaussian. Finally it is possible to prove using similar arguments as in \cite{Cie-Geo-17} that the order of the variance in the flat edge is $o(1)$. 

From these considerations, we expect that the large deviations, for the boundary model, to be `unusual' in the characteristic direction,  while in the off-characteristic directions to be the typical decay of order $e^{-N}$  for both tails. We can show that the right tail has deviations of order $e^{-cN}$, but conditional on one of the boundaries being absent. This is essentially equation \eqref{eq:horj}. In Lemma \ref{lem:pow} we give a bound on the left tail that indicates superexponential decay when we move along direction \eqref{eq:chardirmacro} for the boundary model. 
\end{remark}

\begin{proof}[Proof of Theorem \ref{thm:LLNG}]
From equations \eqref{eq:gradients} we may write the last passage time of the invariant model as 
\[
G^{(u)}_{\fl{Ns}, \fl{Nt}} = \sum_{j=1}^{\fl{Ns}}I^{(u)}_{i, 0} + \sum_{j=1}^{\fl{Nt}}J^{(u)}_{ \fl{Ns},j} 
\]
where the $I, J$ variables are respectively the horizontal and vertical increments of the passage time. 
By the definition of the boundary model, the $I$ variables are i.i.d.\ Ber$(u)$. Scaled by $N$, the first sum converges to $s \E(I_{1,0})$ by the law of large numbers.

By Corollary \ref{cor:downrightpath} the $J$ variables are i.i.d.\ Geom$(\frac{u-p}{u(1-p)})$, since they belong on the down-right path that goes from $(0,\fl{Nt})$ horizontally to $(\fl{Ns},\fl{Nt})$  and then vertically down to $(\fl{Ns},0)$. At this point we cannot immediately evoke the law of large numbers as before since the whole sequence (and therefore the whole second sum) changes with $N$. To show that this does not alter the law of large numbers limit, we first appeal to the Borel-Cantelli lemma via a large deviation estimate. Fix an $\e>0$. 
\begin{align*}
\P\Big\{ N^{-1} \sum_{j=1}^{\fl{Nt}}J^{(u)}_{ \fl{Ns},j} \notin \Big( \frac{p(1-u)}{u-p}&-\e,  \frac{p(1-u)}{u-p}+ \e\Big) \Big\} \\
&=\P\Big\{ N^{-1} \sum_{j=1}^{\fl{Nt}}J^{(u)}_{ 0,j} \notin \Big( t \, \ \frac{p(1-u)}{u-p}-\e, t \, \ \frac{p(1-u)}{u-p}+ \e\Big) \Big\} \\
& \le e^{-c(u, p,t ,\e)N},
\end{align*}
for some proper positive constant $c(u,p, t, \e)$. Then for each $\e>0$ there exists a random $N_{\e}$ so that for all $N > N_\e$
\[
t\, \  \frac{p(1-u)}{u-p} - \e < N^{-1}\sum_{j=1}^{\fl{Nt}}J^{(u)}_{\fl{Ns},j} \le t \, \ \frac{p(1-u)}{u-p} + \e,
\]
from the Borel-Cantelli lemma.
Then we have 
\[
su + t\, \ \frac{p(1-u)}{u-p} - \e \le \varliminf_{N \to \infty} \frac{G^{(u)}_{\fl{Ns}, \fl{Nt}}}{N} \le \varlimsup_{N \to \infty} \frac{G^{(u)}_{\fl{Ns}, \fl{Nt}}}{N} \le su + t\, \ \frac{p(1-u)}{u-p} + \e.
\]
The proof is complete when $\e$ is sent to 0. 
\end{proof}

\begin{proof}[Proof of Theorem \ref{thm:HOR}]
By definition \eqref{lpthor} and \eqref{eq:busopt} we have
\begin{align}
g_{pp}^{(u),\text{hor}}(s,t)&=\lim_{N\to\infty}N^{-1}G^{(u),\text{hor}}_{\fl{Ns},\fl{Nt}}\notag\\
&=\lim_{N\to\infty}\max_{1\leq k\leq\fl{Ns}}\Big\{N^{-1}\sum_{i=1}^kI^{(u)}_{i,0}+N^{-1}G_{(k,1),(\fl{Ns},\fl{Nt})}\Big\}\label{eq:uguale}\\
&=\sup_{0\leq a\leq s}\{au+g_{pp}(s-a,t)\}.\notag
%&=\sup_{0\leq a\leq s}\notag
%\begin{cases}
%au+(\sqrt{p(s-a)}+\sqrt{(1-p)t})^2-t, \qquad &t<(s-a)\frac{1-p}{p},\\
%au+s-a, \qquad &t\geq(s-a)\frac{1-p}{p}.
%\end{cases}
\end{align}
The last line follows by the same coarse graining arguments as in the proof of Theorem \ref{thm:llnp} (see Appendix \ref{app:	 LLN}). 

If $t<\frac{1-p}{p}s$
\[
g_{pp}^{(u),\text{hor}}(s,t)=\sup_{0\leq a\leq s - \frac{pt}{1-p}}\{ au+(\sqrt{p(s-a)}+\sqrt{(1-p)t})^2-t \} \vee \sup_{s - \frac{pt}{1-p} <a \le s}\{ a(u-1) + s \}. 
\]
The second supremum is attained at the boundary point $s - \frac{pt}{1-p}$ since it optimizes a decreasing function of $a$. 
In the first supremum, a unique minimizing point exists and it is either a boundary point or the critical point $a^*$ of the derivative of  $f(a)=au+(\sqrt{p(s-a)}+\sqrt{(1-p)t})^2-t$, given by 
\[
 a^*=s-\frac{p(1-p)t}{(u-p)^2}.
\]
If $s-\frac{p(1-p)t}{(u-p)^2}<0$ then we have that $a^*=0$. Otherwise, note that $a^*$  we can substitute $a^*$ into $f(a)$ and obtain
\[f(a^*)=su+\frac{p(1-u)}{u-p}t=g^{(u)}_{pp}(s,t).\]

Finally, if $t\geq\frac{1-p}{p}s$
\[
g_{pp}^{(u),\text{hor}}(s,t)=\sup_{0\leq a\leq s}\{au+s-a\}=s=g_{pp}(s,t).
\]
The proof for $g_{pp}^{(u),\text{ver}}(s,t)$ is similar and left to the reader.
\end{proof}

\section{I.i.d.\ Model: Full LDP}
\label{sec:fullLDP}

We first focus on the model without boundaries. Recall that the maximal path can collect Bernoulli weights only when it takes a step to the right. 
The full rate function is described in Theorem \ref{thm:fullLDP}. As it is usually the case with models of last passage percolation, large deviations of the passage time above its mean are of different exponential scale than the deviations below its mean. With this in mind, in order to obtain a full LDP, one only needs the right-tail rate function.  
This is our beginning point.

Suppose that the target point is $(s,t)$, then, since the last passage time collects  Bernoulli weights only through the right step, \eqref{rval} implies that the probability 
\be\label{rval}
\P\{ G_{\fl{Ns}, \fl{Nt}} \ge Nr \}  \neq 0 \text{ if and only if } r < s.
\ee
In the particular case where $s$ is rational, the probability above can be strictly positive for certain values of $N$, but otherwise it is $0$. 

\begin{theorem}\label{JInt}
For $((s,t), r)$ with $0 \le r < s < \infty$ and $t \in \R_+$, the following $\R_+$-valued limit exists:
\be\label{J}
-\lim_{N\to\infty}N^{-1}\log\P\{G_{\fl{Ns}, \fl{Nt}}\geq Nr\} = J_{s,t}(r).
\ee
$J_{s,t}(r)$, as a function of $((s,t), r )$ is a continuous convex function on the interior of the set $A = \{ ((s, t), r): s \ge r \vee 0, t \in \R_+, r \in \R_+\}$.
It can be uniquely extended to a finite continuous convex function on the closure $\bar A$.
Moreover,  $J_{s,t}(r) > 0$ for $r > g_{pp}(s,t).$
\end{theorem}

\begin{remark}
From this point onwards, we assume that $J_{s,t}(r)$ will be the function with domain $\bar A$, to have a generalised lower-semicontinuous convex function on $\R^3$. The only thing one needs to keep in mind is that when we are discussing boundary values, the limit representation \eqref{rval} is no longer available to us. This does not affect the results that follow.
\end{remark}

We present the proof of this basic result in Appendix \ref{app:prop}. The methodology utilises the super-additivity of $G$ and follows the proof steps of \cite{Sep-98-mprf-2}.

\begin{figure}
\begin{center}
\includegraphics[height=4.5cm]{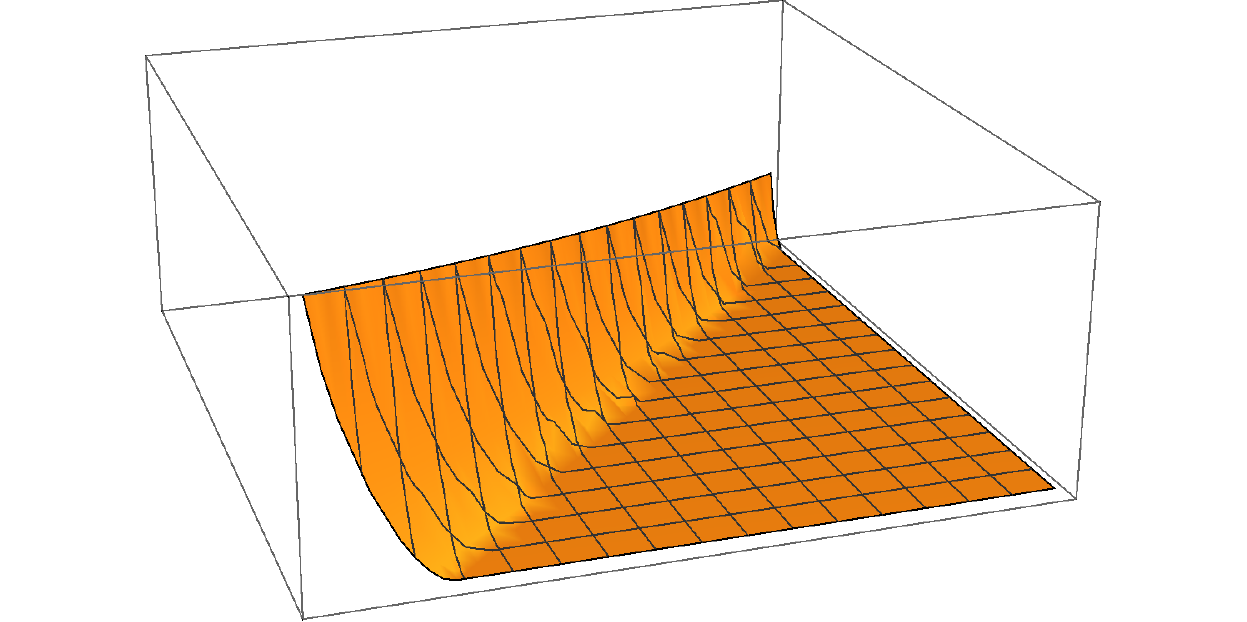}
\hspace{0cm}
\includegraphics[height=3.6cm]{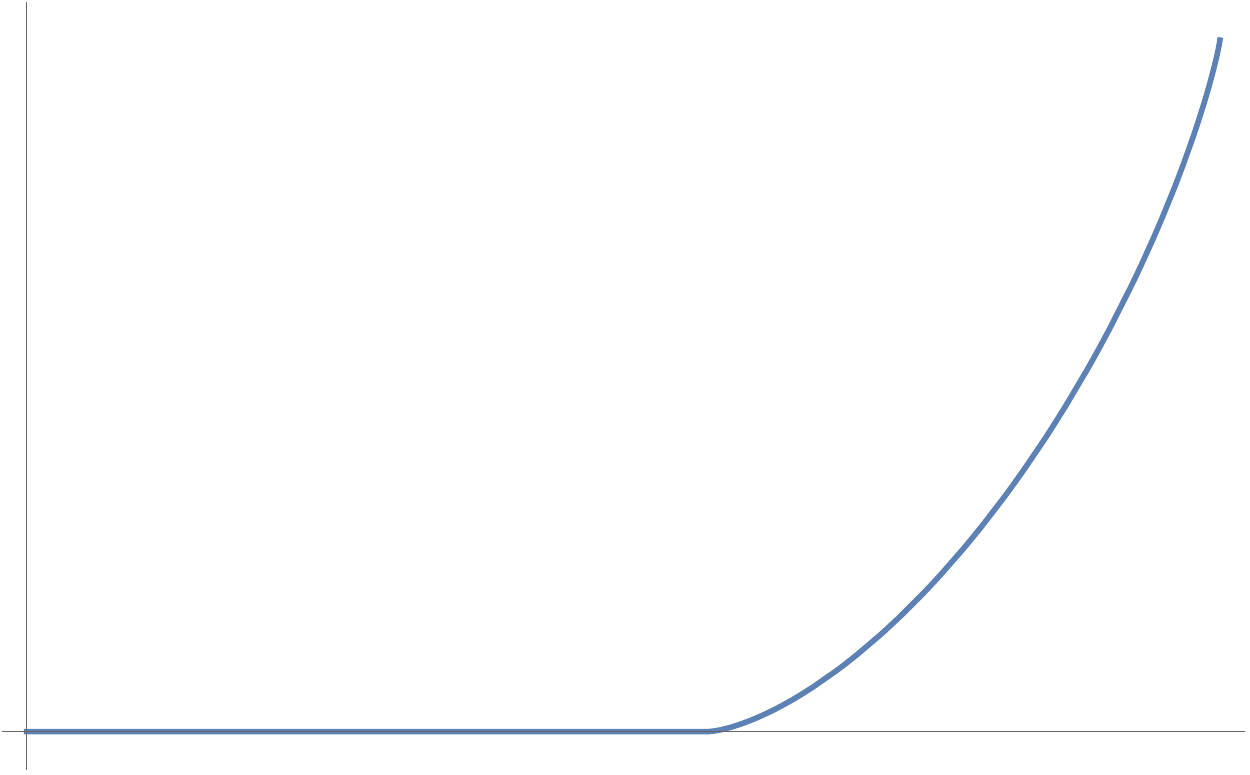}
\end{center}
\caption{Graphical representation of the function $J_{s,t}(r)$. In both figures we used $p=1/2$ and $t=1$. To the left we have the lower-semicontinuous version of $J_{s,1}(r)$ as a function of $(s,r)$. You can see that it is finite for $s \le r$. To the right is the function $J_{2,1}(r)$.}
\end{figure}

The continuous extension up to $\bar A$ makes the function $J_{s,t}(r)$ lower-semicontinuous on $\R^3$ where it takes the value $\infty$ outside of $\bar A$.

It will be useful to also know some of the boundary values of the lower semi-continuous extension. 
We summarise the results in the following corollary: 
\begin{corollary}\label{thm:JJJ} The lower-semicontinuous extension of $J_{s,t}(r)$ takes the following values on $\partial A$
\be
J_{s,t}(r) = 
\begin{cases} 
0, &\quad t = 0, r \le 0, s \in \R_+,\\
0, &\quad t \in \R_+, r \le 0, s=0, \\
sI^{(p)}_{\mathcal B}(r/s), & \quad  t = 0, 0 \le r \le s, s \in \R_+\\
\lim_{r \nearrow s} J_{s,t}(r), & \quad  t, s \in \bR_>0,  r = s.
\end{cases}
\ee 
\end{corollary}
This follows from the fact that $J_{s,t}$ needs to be lower-semicontinuous, and it is briefly discussed after the proof of Theorem \ref{JInt} in the Appendix. 
Above we defined  $I _{\mathcal B}$ to be the Cram\'er rate function for sums of i.i.d. $\om_i \sim$ Bernoulli$(p)$,
 \be\label{eq:berrate}
I^{(p)}_{\mathcal B}(r) =  
\begin{cases}
-\displaystyle\lim_{N\to\infty}N^{-1}\log\P\Big\{ \sum_{i=1}^{N}\om_{i}\geq N r \Big\} = r\log \frac{r}{p}+(1-r)\log\frac{1-r}{1-p}, & r \in[p, 1],\\
0, &r < p\\
\infty, & r> 1.
\end{cases}
 \ee

Finally, we show the existence of a good rate function $I_{s,t}(r)$ and list its properties; this is the content of the next theorem. 
We restrict $r\in[0,s]$ because $I_{s,t}(r)= \infty $ for any $r$ outside this interval. 

\begin{theorem}\label{thm:fullLDP}
Let $\om_{i,j}\sim$Bernoulli$(p)$ with $i,j\geq1$ and $(s, t)\in(0,\infty)^2$. Then there exists a generalised function $I_{s, t}(r)$ so that the  distributions of $N^{-1}G_{\fl{Ns},\fl{Nt}}$ satisfy an LDP with normalisation $N$ and rate function $I_{s,t}(r)$. 
To be precise, the following bounds hold for any open set $H$ and any closed set $F$ in $[0,s]$:
\be\label{upldp}
\varlimsup_{N\to\infty} N^{-1} \log \P\{N^{-1} G_{\fl{Ns},\fl{Nt}}\in F \} \leq - \inf_{r\in F} I_{s,t}(r) 
\ee
and
\be\label{loldp}
\varliminf_{N\to\infty} N^{-1} \log \P\{N^{-1}  G_{\fl{Ns},\fl{Nt}}\in H \} \geq - \inf_{r\in H} I_{s,t}(r) .
\ee

The rate function $I_{s,t}(r)$ is defined by 
\be\label{eq:III}
I_{s,t}(r)=
\begin{cases}
J_{s,t}(r),\qquad &r\in[g_{pp}(s,t),s],\\
\infty,\qquad &\text{otherwise}.
\end{cases}
\ee
Rate function $J_{s,t}(r)$ is the right-tail rate function computed in Theorem \ref{JInt}. In particular,  
on $[g_{pp}(s,t),s]$ the rate function $I_{s,t}$ is finite, strictly increasing, continuous and convex. Moreover, the unique zero of $I_{s,t}(r)$ is at $r = g_{pp}(s, t)$. 
\end{theorem}

In order to obtain a full large deviation principle, we must estimate the lower tail for the probabilities of the last passage time. As is usual in the solvable models of last passage percolation, the speed for the lower tail is different than $N$. Our first lemma establishes the same fact for this model.  In turn, this gives left tail bounds strong enough to imply $I_{s,t}(r)=\infty$ for $r<g_{pp}(s,t)$ for both boundary and i.i.d.\ model.

\begin{lemma}\label{lem:pow}
There exist constants $c>0$, $C<\infty$ that depend on parameters $s, t, p, u$,so that for all $N \ge 1$ the following estimates hold:
\begin{enumerate}[(a)]
\item For $(s,t)\in(0,\infty)^2$ and $r\in[0,g_{pp}(s,t))$
\[
\P\{G_{\fl{Ns},\fl{Nt}}\leq Nr\}\leq Ce^{- cN^{2}}.
\]
\item For $(s,t)=\alpha\Big(1,\frac{(u-p)^2}{p(1-p)}\Big)$ for some $\alpha>0$, parallel to the characteristic direction, and $r\in[0,g^{(u)}_{pp}(s,t))$,
\[
\P\{G^{(u)}_{\fl{Ns},\fl{Nt}}\leq Nr\}\leq Ce^{- cN^{2}}.
\]
\end{enumerate}
\end{lemma}
%\nicos{check (b) for outside the characteristic direction? Can we get CLT type fluctuations?}. 

\begin{proof}[Proof of Lemma \ref{lem:pow}]

We prove (b) but similar arguments work for (a). We bound $G^{(u)}_{\fl{Ns},\fl{Nt}}$ from below, using the superadditivity property of the last passage times. 

For this reason we consider a subset of lattice paths, arranged in a collection of i.i.d.\ partition function over subsets of rectangles. This block argument proof was first used in \cite{Kes-86-aspects} and later adapted in \cite{Sep-98-mprf-2} for the last passage time and in \cite{Geo-Sep-13-, Ja-15} for the log-gamma polymer and the brownian polymer model respectively.

Note that if $(s,t)$ are chosen in the characteristic direction it is immediate to see that $g^{(u)}_{pp}(s,t)=g_{pp}(s,t)$. 

We first show the result for $(s,t) \in \Q_+^2$. In order to highlight this distinction we assume that the target point is $(q_1, q_2) \in \Q^2_+$.
 Fix $0<\e<1/4(g^{(u)}_{pp}(q_1,q_2)-r)$. Define a new scale parameter $m\in\N$ large enough so that $m(q_1\wedge q_2)\geq1$, $ mq_1,mq_2\in \N_+$ and 
\be\label{mE}
\E[G_{mq_1,mq_2}] >m(r+2\e).
\ee

We will use $mq_1$ and $mq_2$ to coarse-grain our environment. %If $(q_1,q_2)$ are chosen in the characteristic direction it is immediate to see that $g^{(u)}_{pp}(q_1,q_2)=g_{pp}(q_1,q_2)$.
Let $\mathcal R^{k,\ell}_{a,b}=\{a,\dots,a+k-1\}\times\{b,\dots,\ell+b-1\}$ denote the $k\times\ell$ rectangle with lower left corner at $(a,b)$. For $i,\ell\geq0$ define pairwise disjoint $mq_1\times mq_2$ rectangles 
\[
\mathcal R^i_\ell=\mathcal R^{mq_1,mq_2}_{(\ell+i)mq_1+i+1,\ell(mq_2+1)+1}.
\] 
The rectangles $\mathcal R^i_\ell$ are separated by the inter-site distance to avoid a scenario where a path goes along a common edge  between two rectangles. This way,  we will be able to clearly say in which one of the two rectangles the path goes through. For each $i$ we define the diagonal union of rectangles as $\Delta_i=\bigcup_{\ell\geq0}\mathcal R^i_\ell,i\geq0$ and in the sequence we are considering potential paths that stay in a fixed $\Delta_i$.

Moreover, note that the last passage times $G_{\mathbf{v}^w_{i,\ell},\mathbf{v}^e_{i,\ell}}$ in each rectangle are all identically distributed, where $\mathbf{v}^w_{i,\ell}=(\ell+i)mq_1+1+i,\ell(mq_2+1)+1)$ and $\mathbf{v}^e_{i,\ell}=((i+1)mq_1+\ell mq_1,(1+\ell)(mq_2+1))$ are respectively the south-west and north-east corners of $\mathcal R^i_\ell$. 

Define $B,M = M(B)\in \N$ the maximal integers which satisfy
\begin{align}
&(M+1)mq_1+Bmq_1\leq Nq_1\qquad\text{and}\label{eq:fin}\\
&(1+B)(mq_2+1)\leq Nq_2.\label{eq:sin}
\end{align}
The fact that $B$ is maximal and \eqref{eq:sin} imply that 
\be\label{eq:B}
B=\fl{\frac{Nq_2}{mq_2+1}}-1.
\ee
Substituting \eqref{eq:B} in \eqref{eq:fin} we obtain
\be\label{eq:M}
M= \fl{\frac{N}{m}-\fl{\frac{Nq_2}{mq_2+1}}}\text{ and hence } \fl{\frac{N}{m(mq_2+1)}}\leq M\leq \fl{N\frac{mq_2+2}{m(mq_2+1)}}.
\ee
Since $m$ is a constant and assumed much smaller than $N$, we have that $B=B(N)=\mathcal O(N)$ and $M=M(N)=\mathcal O(N)$. Fix a diagonal $\Delta_i$ for $0\leq i\leq M$ and define the union of rectangles in $\Delta_i \cap ([0,Nq_1]\times[0,Nq_2])$ as $\Delta_i^{B}=\bigcup_{0\leq\ell\leq B}\mathcal R^i_\ell$. 

%If the expressions inside the floor functions in \eqref{eq:B} and \eqref{eq:M} are not integers we have that the point $(Nq_1,Nq_2)$ doesn't coincide with the north-east corner of the rectangle $\mathcal R^M_B$. Thus it is possible that some lattice points are outside the region $\Delta_i^{B}$. But since they are all positive random variable we can neglect them finding an upper bound to our initial probability.

\begin{figure}
\begin{center}
	\begin{tikzpicture}[>=latex, scale=0.27]

\draw[fill=blue!20,blue!20] (15,0.75)--(15,-0.75)--(-0.75,-0.75)--(-0.75,0.75)--(15,0.75);
\draw (15,0.75)--(-0.75,0.75)--(-0.75,-0.75)--(15,-0.75);
\draw (0,0)--(15,0);
\draw (0,0)--(0,11);
\foreach \x in {0,1,...,8}{
\draw[dashed] (15,\x*3)--(19,\x*3);}
\foreach \x in {0,1,...,4}{
\draw[dashed] (\x*3,11)--(\x*3,15);}
\foreach \x in {7,...,11}{
\draw[dashed] (\x*3,11)--(\x*3,15);}
\draw[->] (19,0)--(35,0);
\draw[->] (0,15)--(0,27);
\draw (33,11)--(33,0)node[below]{\tiny$\fl{Ns}$};
\draw (15,25.5)--(0,25.5)node[left]{\tiny$\fl{Nt}$};
\draw (19,25.5)--(33,25.5);
\draw (33,15)--(33,25.5);
\draw[dashed] (15,25.5)--(19,25.5);
\foreach \x in {1,2,...,10}{
\foreach \y in {1,2,...,7}
          {\shade[ball color=gray](1.5*\x,1.5*\y) circle (1.5mm);}}
\foreach \x in {13,14,...,22}{
\foreach \y in {1,2,...,7}
          {\shade[ball color=gray](1.5*\x,1.5*\y) circle (1.5mm);}}
\foreach \x in {1,2,...,10}{
\foreach \y in {11,12,...,17}
          {\shade[ball color=gray](1.5*\x,1.5*\y) circle (1.5mm);}}
\foreach \x in {13,14,...,22}{
\foreach \y in {11,12,...,17}
          {\shade[ball color=gray](1.5*\x,1.5*\y) circle (1.5mm);}}
\foreach \y in {11,12,...,17}
          {\shade[ball color=blue](0,1.5*\y) circle (1.5mm);}
\foreach \y in {1,2,...,7}
          {\shade[ball color=blue](0,1.5*\y) circle (1.5mm);}
\foreach \x in {1,2,...,10}
          {\shade[ball color=nicos-red](1.5*\x,0) circle (1.5mm);}
\foreach \x in {13,14,...,22}
          {\shade[ball color=nicos-red](1.5*\x,0) circle (1.5mm);}
\draw[fill=blue!20] (1.5,1.5)--(6,1.5)--(6,4.5)--(1.5,4.5)--(1.5,1.5);
\draw[fill=blue!20] (7.5,1.5)--(12,1.5)--(12,4.5)--(7.5,4.5)--(7.5,1.5);
\draw[fill=blue!20] (6,6)--(10.5,6)--(10.5,9)--(6,9)--(6,6);
\draw[fill=blue!20,blue!20] (12,6)--(15,6)--(15,9)--(12,9)--(12,6);
\draw[fill=blue!20,blue!20] (13.5,1.5)--(15,1.5)--(15,4.5)--(13.5,4.5)--(13.5,1.5);
\draw (15,6)--(12,6)--(12,9)--(15,9);
\draw (15,1.5)--(13.5,1.5)--(13.5,4.5)--(15,4.5);
\draw[fill=blue!20] (30,22.5)--(25.5,22.5)--(25.5,19.5)--(30,19.5)--(30,22.5);
\draw[fill=blue!20] (24,22.5)--(19.5,22.5)--(19.5,19.5)--(24,19.5)--(24,22.5);
\draw[fill=blue!20,blue!20] (25.5,18)--(21,18)--(21,16.5)--(25.5,16.5)--(25.5,18);
\draw (25.5,16.5)--(25.5,18)--(21,18)--(21,16.5);

\draw (3.5,3)node{\tiny$\mathcal R_0^0$};
\draw (10,3)node{\tiny$\mathcal R_0^1$};
\draw (8.5,7.5)node{\tiny$\mathcal R_1^0$};
\draw (14,7.5)node{\tiny$\mathcal R_1^1$};
\draw (28,21)node{\tiny$\mathcal R^M_B$};
\draw (22,21)node{\tiny$\mathcal R^{M-1}_B$};
\draw (28,23.2)node{\tiny$\Delta_M$};
\draw (22,23.2)node{\tiny$\Delta_{M-1}$};
\draw (8.5,9.7)node{\tiny$\Delta_0$};
\draw (14,9.7)node{\tiny$\Delta_{1}$};
\draw[line width=2pt,blue](0,0)--(9,0)--(9,4.5)--(12,4.5)--(12,9)--(15,9);
%\draw[line width=2pt,blue](19,19.5)--(19.5,19.5)--(19.5,22.4)--(24,22.5);
\draw[line width=2pt,blue](19,24)--(33,24)--(33,25.5);
\draw[my-green, line width=1.5pt] (4.55,0) ellipse (160pt and 30pt);
\draw[my-green, line width=1.5pt]  (26.5,25.25) ellipse (200pt and 45pt);
\end{tikzpicture}
\end{center}
\caption{Representation of the coarse grained $\fl{ms}\times \fl{mt}$ rectangles and the diagonals $\Delta_i$ in the proof of Lemma \ref{lem:pow}. The blue thick line is one of the possible maximal paths. For the bound needed, we are allowed to ignore the path segments outside of the coarse-grained diagonals, particularly  we may ignore the correlated segments when candidate paths traverse the south and north boundary of $[0 ,\fl{Ns}]\times[0, \fl{Nt}]$. Passage times in each $\Delta_1$ are i.i.d. and smaller than the overall passage time.}
\label{fig:1}
\end{figure}

Let $G^\Delta_i$ be the last passage time of all lattice paths in $\Delta_i^{B}$ from the lower left corner of $\mathcal R_0^i$ to the upper right corner of $\mathcal R_B^i$. $G^\Delta_i$ are i.i.d, where in particular $G^\Delta_0$ is the sum of the $B$ last passage times of rectangle $\mathcal R^0$ whose mean is controlled by \eqref{mE}. A standard large deviation estimate for an i.i.d sum gives the following bound 
\be\label{sp}
\begin{aligned}
\P\{G^{(u)}_{\fl{ Nq_1}, \fl{ Nq_2}}\leq Nr\}&\leq \P\{G^\Delta_i\leq Nr \text{ for }0\leq i\leq M\}\\
&\leq \P\{G^\Delta_0\leq Nr\}^M \leq\P\Big\{\sum_{k=0}^{B(N)} G_k^0\leq Nr\Big\}^{M(N)}\\
&\leq e^{-cB(N)M(N)}\leq e^{-c_1N^{2}}.
\end{aligned}
\ee
This completes the proof for $(s, t) \in \Q_+^2$. 

Finally we show \eqref{sp} holds also for $s,t\in \R_+$. We bound $G^{(u)}_{\fl{Ns}, \fl{Nt}}$ using $G^{(u)}_{\fl{ Nq_1}, \fl{Nq_2}}$ for some special $(q_1,q_2)\in\Q^2_+$ which are close enough to $(s,t)\in \R^2_+$. For any $(q_1, q_2) \le (s,t)$ we have that 
\be\label{f}
\P\{G^{(u)}_{\fl{Ns},\fl{Nt}}\leq Nr\}\leq \P\{G^{(u)}_{\fl{Nq_1}, \fl{Nq_2}}\leq Nr\}, \text{ for all }  r \in [0, g^{(u)}_{pp}(s,t)).
\ee
For any $\delta > 0$ and find $(q_1, q_2)$ so that $\delta>g^{(u)}_{pp}(s,t)- g_{pp}^{(u)}(q_1,q_2) >0$. 
This is possible by the continuity and monotonicity of $g^{(u)}_{pp}$. We choose $\delta < \frac{g^{(u)}_{pp}(s,t) - r}{2}$ and therefore
\[
r <  g^{(u)}_{pp}(s,t) -2\delta < g_{pp}^{(u)}(q_1,q_2) - \delta < g^{(u)}_{pp}(s,t),
\]
for some $(q_1, q_2) \in \Q^2_+$
Then \eqref{f} is a left-tail large deviation anyway for $G^{(u)}_{\fl{Nq_1},\fl{Nq_2}}$ so \eqref{sp} holds.
\end{proof}

\begin{proof}[Proof of Theorem \ref{thm:fullLDP}]
This proof is a consequence of the lemmas and theorems that we have already proved. Define for $r\in\R$ function $ I_{s, t}(r)$ by \eqref{eq:III}.

Then, the regularity properties proved for $J$ in Theorems \ref{JInt} and \ref{thm:JJJ} are also valid for $I_{s,t}$. For the upper large deviation bound \eqref{upldp} we consider two cases:
\begin{enumerate}[(1)]
\item if $F\subseteq [0,g_{pp}(s,t))$, then $r^* = \max \{ x: x \in F\}  < g_{pp}(s,t)$  and we have 
\[
\P\{N^{-1}G_{\fl{Ns},\fl{Nt}}\in F\}\leq \P\{G_{\fl{Ns},\fl{Nt}}\leq Nr^*\}\leq e^{-N^2}.
\]
The last inequality comes from Lemma \ref{lem:pow}. Take logarithms on both sides, divide by $N$, take the limit $N\to \infty$ and finally by definition \eqref{eq:III} conclude that 
\[
\lim_{N\to\infty}N^{-1}\log\P\{N^{-1}G_{\fl{Ns},\fl{Nt}}\in F\}= -\infty=-\inf_{r\in F}I_{s,t}(r).
\]
\item If $F\cap[g_{pp}(s,t),s]\not=\emptyset$ then we split into two different cases: 
\begin{enumerate}[]
\item Case 1: $F \not \ni g_{pp}(s,t)$. Then there exists an $\e>0$ such that $(g_{pp}(s,t)-\e, g_{pp}(s,t)+\e) \subseteq F^c$. Then we bound 
\begin{align*}
\P\{N^{-1}&G_{\fl{Ns},\fl{Nt}}\in F\} \le \P\{N^{-1}G_{\fl{Ns},\fl{Nt}} \le g_{pp}(s, t) - \e \} \\
&\phantom{xxxxxxxxxxxxxxx}+\P\{N^{-1}G_{\fl{Ns},\fl{Nt}}\in F\cap[g_{pp}(s,t)+\e,s]\}.
\end{align*}
By the previous calculations, we already control the first addend by $e^{-cN^2}$ therefore we focus only on the second one which will be of an exponential order of magnitude larger and control the value of the $\varlimsup$. 
Since $F$ and $[g_{pp}(s,t)+\e,s]$ are two closed sets there exists an $r^*$ such that
$
r^*=\min\{r: r\in F\cap[g_{pp}(s,t)+\e,s]\}.
$
It follows that 
\[
\P\{N^{-1}G_{\fl{Ns},\fl{Nt}}\in F\}\leq e^{-cN^2}+\P\{G_{\fl{Ns},\fl{Nt}}\geq Nr^*\}.
\]
Now take the logarithm of both sides, divide by $N$ and take the $\varlimsup$,
\begin{align*}
\lim_{N\to\infty}N^{-1}\log\P\{N^{-1}G_{\fl{Ns},\fl{Nt}}\in F\}& \leq \lim_{N\to\infty}N^{-1}\log( e^{-cN^2} + \P\{G_{\fl{Ns},\fl{Nt}}\geq Nr^*\})\\
&=- J_{s,t}(r^*) = -\inf_{r\in F}I_{s,t}(r).
\end{align*}
The last line is obtained using \eqref{J}, \eqref{eq:III} and the fact that $I_{s,t}(r)$ is a strictly increasing function. 

\item Case 2: $F  \ni g_{pp}(s,t)$. In this case, $\inf_{r \in F} I_{s,t}(r) = 0$, therefore,  inequality \eqref{upldp} is automatically satisfied.
\end{enumerate}

\end{enumerate}

 For the lower large deviation bound \eqref{loldp}, we need to consider three cases according to $H$:
\begin{enumerate}[(1)]
\item If $g_{pp}(s, t) \in  H$, then $\P\{N^{-1}G_{\fl{Ns},\fl{Nt}}\in  H\} \to 1$ and \eqref{loldp} holds as an equality

\item If $H \subseteq [0, g_{pp}(s, t))$, \eqref{loldp} holds because its right-hand side is $-\infty$.

\item The remaining case is the one where $H$ contains an interval $(a, b) \subset (g_{pp}(s,t),s)$. %We differentiate between $a = g_{pp}(s,t)$ or$ a >  g_{pp}(s,t)$.
 Then for any $\e>0$  small enough, we can find a non-trivial interval $[a + \e, b - \e] \subseteq H$ and bound
\begin{align}
N^{-1}\log &\P\{N^{-1} G_{\fl{Ns},\fl{Nt}}\in H \}\geq N^{-1}\log \P\{N^{-1} G_{\fl{Ns},\fl{Nt}}\in [a+\e,b-\e] \} \notag \\
%&=N^{-1}\log \Big( \P\{N^{-1} G_{\fl{Ns},\fl{Nt}}\leq N(b-\e) \}-\P\{N^{-1} G_{\fl{Ns},\fl{Nt}}\leq N(a+\e)\}\Big) \notag\\
&=N^{-1}\log \Big(\P\{N^{-1} G_{\fl{Ns},\fl{Nt}}\geq N(a+\e)\}-\P\{N^{-1} G_{\fl{Ns},\fl{Nt}}\geq N(b-\e)\Big) \notag\\
%&\geq N^{-1}\log \P\{N^{-1} G_{(\fl{Ns},\fl{Nt})}\geq Na\}+  N^{-1}\log \P\{N^{-1} G_{(\fl{Ns},\fl{Nt})}\geq Nb\}\label{1}\\
& \to-J_{s,t}(a+\e).\label{eq:nec2}
\end{align}
Equation \eqref{eq:nec2} follows after taking $\varliminf$ on both sides and keeping in mind that the two terms in the logarithm have different exponential orders of magnitude.  

Monotonicity and convexity $J_{s,t}$ on $[g_{pp}(s,t),s]$ implies that for some constant $C$,
$
J_{s, t}(a + \e) \leq J_{s,t} (a) + C\e.
$ 
Then, \eqref{eq:nec2} becomes 
\[
\varliminf_{N \to \infty} N^{-1}\log \P\{N^{-1} G_{\fl{Ns},\fl{Nt}}\in H \} \geq - J_{s,t} (a) - C\e
\]
Let $\e \to 0$ in the last display. Then take $a=\inf H\cap (g_{pp}(s,t),s)$ to finish using  
\[ J_{s,t}(a)=\inf_{r\in H\cap (g_{pp}(s,t),s)} I_{s,t}(r) = \inf_{r\in H} I_{s,t}(r). \qedhere \]
\end{enumerate} 
\end{proof}

\begin{corollary}\label{cor:lam}
Let $\xi\in\R$. Then 
\be\label{lamxi}
\lim_{N\to\infty}N^{-1}\log\E e^{\xi G_{0,(\fl{Ns},\fl{Nt})}} =I^*_{s,t}(\xi)=\begin{cases}
J_{s,t}^*(\xi)\quad &\text{if }\xi>0,\\
0\quad &\text{if }\xi=0,\\
\xi g_{pp}(s,t)\quad &\text{if }\xi<0,
\end{cases}.
\ee
\end{corollary}

\begin{proof}[Proof of Corollary \ref{cor:lam}]
Since $G_{\fl{Ns},\fl{Nt}} \le Ns$, for any $\gamma > 1$ and $\xi\in \R$,
\[
\sup_{N}\Big(\E e^{\gamma\xi\, G_{\fl{Ns},\fl{Nt}}}\Big)^{1/N}<\infty.
\]
This bound together with Theorem \ref{thm:fullLDP} suffice to apply Varadhan's theorem (e.g.\ page 38 in \cite{LDP-Ras-Sep-16}) which gives
\begin{align*}
\lim_{N\to\infty}N^{-1}\log\E e^{\xi G_{0,(\fl{Ns},\fl{Nt})}}&=I^*_{s,t}(\xi)=\sup_{r\in \R}\{r\xi-I_{s,t}(r)\}\\
&=\sup_{r\in[g_{pp}(s,t),s]}\{r\xi-I_{s,t}(r)\}=\sup_{r\in[g_{pp}(s,t),s]}\{r\xi-J_{s,t}(r)\}.
\end{align*}
The first equality on the second line is because $I_{s,t}(r)=\infty$ if $r\in (-\infty,g_{pp}(s,t))$ or $r > s$ 
and there is no difference in excluding that interval from the supremum. 

Then we can compute $I^*_{s,t}$.  
$I_{s,t}$ is increasing for $r\in[g_{pp}(s,t),s]$, therefore if $\xi<0$, the supremum is always attained at $r=g_{pp}(s,t)$. 
Instead, when $\xi\geq0$, $I_{s,t}(\xi)^*=J^*_{s,t}(\xi)$ since $r$ can range over all of $\R$ and the last supremum will still be attained for some
$r \in [g_{pp}(s,t), s]$.
\end{proof}

%%%%%%%%%%%%%%%%%%%%%%%%%%%%%%%%%%%%%%%%%%%%%%%%%%%%%%%%%

\section{I.i.d.\ model: Right tail rate function and log moment generating function}
\label{sec:ldpwb}

The main goal of this section is to prove an explicit variational formula for the rate function $J_{s,t}(r)$. 
That formula, while precise does not enjoy enough analytical tractability to further obtain a closed formula. 
However, its dual $J^*_{s,t}(\xi)$ will be explicitly computed by the end of the section. The variational characterization
 of $J$ requires the log-moment generating functions 
for Bernoulli$(p)$ random variables, given by   
\be\label{eq:bercum}
C_{\mathcal B}^{(p)}(\xi) = \log(1 - p + p e^{\xi}), \xi \in \R.
\ee
and Geometric$(p)$ random variables given by 
\be\label{eq:geocum}
C_{\mathcal G}^{(p)}(\xi) = \begin{cases}
\log\frac{p}{1-(1-p)e^\xi}, &  \xi <-\log(1-p)\\
\infty, &\text{ otherwise}.
\end{cases}
\ee
Both log-moment generating functions can be seen as the Legendre duals of the rate functions  
for sums of i.i.d. Bernoulli \eqref{eq:bercum} and for sums of i.i.d. geometric random variables, given by 
\be\label{eq:georate}
I^{(p)}_\mathcal G(r)=\sup_{\xi<-\log(1-p)}\big\{r\xi-C_{\mathcal G}^{(p)}(\xi) \big\}=r\log\frac{r}{(1-p)(1+r)}-\log(1+r)p\qquad \text{for } r>0.
\ee
The two theorems that give the precise forms for 
$J$ and $J^*$ follow.

\begin{proposition}\label{propJ}
Let $(s,t)\in\R^2_+$. Then for all $\xi \in \R$, the convex dual $J^*_{(s,t)}(\xi)$ is given by
\[
J^*_{s,t}(\xi) = \begin{cases}
\inf_{u\in(p,1]}\{sC_\mathcal B^{(u)}(\xi)-tC_\mathcal G^{(\frac{u-p}{u(1-p)})}(-\xi)\},\quad &\text{if }\xi>0,\\
0,\quad &\text{if }\xi=0,\\
\infty, \quad &\text{if }\xi<0.
\end{cases}
\]
%and for any $r \in [0, s]$, 
%\be\label{JG}
%J_{s,t}(r)=\sup_{\xi\in\R}\{r\xi- J^*_{s,t}(\xi)\}. 
%\ee
\end{proposition}

The closed form for $J^*$ is given in the following 

\begin{theorem}\label{thm:lam}
Fix $p\in(0,1)$, $\xi\geq0$ and $(s,t)\in\R^2_+$. Define
\be
\Delta = \Delta_{p,s, t, \xi}=p(1-p)(e^\xi+e^{-\xi}-2)\big[p(1-p)(s+t)^2(e^\xi+e^{-\xi}-2)+4st\big].
\ee
Then, 
\begin{align}
&J^*_{s,t}(\xi)=\begin{cases}
\begin{aligned}
&s\log\frac{p(1-p)(s+t)(e^\xi+e^{-\xi}-2)+2s+\sqrt{\Delta}}{2s(1-p(1-e^{-\xi}))}\\
&\hspace{0.3cm}+t\log\frac{[p(1-p)(s+t)(e^\xi+e^{-\xi}-2)+\sqrt{\Delta}](1-p(1-e^{-\xi}))}{p(1-p)(t-s)(e^\xi+e^{-\xi}-2)+\sqrt{\Delta}},
\end{aligned} \quad &\text{if }t<\frac{1-p}{p}s,\\
 s\xi,\quad &\text{if }t\geq\frac{1-p}{p}s.
\end{cases}\label{exactform}
\end{align}
\end{theorem}

\subsection{Exact computations for $J_{s,t}(r)$}
\label{subsec:ecom} 
We first present a series of key technical lemmas, and we encourage the reader familiar with these techniques to proceed to the proof of Proposition \ref{propJ}.

We will use the invariance property of the model with boundaries first. Consider the 
last passage time in the model with boundary $G^{(u)}_{\fl{Ns},\fl{Nt}}$ 
and we iteratively apply equation \eqref{eq:gradients} to obtain
\[ 
 G^{(u)}_{\fl{Ns},\fl{Nt}}-G^{(u)}_{0,\fl{Nt}}=\sum_{i=1}^{\fl{Ns}}I^{(u)}_{i,\fl{Nt}}.
\]
%Note that on the right hand side we have a sum of i.i.d.\ Bernoulli distributed with parameter $u$ whose large deviation rate function is known. %For clarity we omit the superscript $(u)$ from the gradients $I$ and $J$ 
Focus on the left hand side. From equation \eqref{eq:varform} and \eqref{eq:gradients} we can write the previous difference as
\begin{align*}
\sum_{i=1}^{\fl{Ns}}I^{(u)}_{i,\fl{Nt}} &=G^{(u)}_{\fl{Ns},\fl{Nt}}-G^{(u)}_{0,\fl{Nt}}\\
&=\max_{1\leq k\leq \fl{Ns}}\Big\{\sum_{i=1}^kI^{(u)}_{i,0}+G_{(k,1),(\fl{Ns},\fl{Nt})}-\sum_{j=1}^{\fl{Nt}}J^{(u)}_{0,j}\Big\}\\
&\hspace{1cm}\bigvee\max_{1\leq k\leq \fl{Nt}}\Big\{\sum_{j=1}^kJ^{(u)}_{0,j}+\om_{1, k} + G_{(1,k),(\fl{Ns},\fl{Nt})}-\sum_{j=1}^{\fl{Nt}}J^{(u)}_{0,j}\Big\}\\
&=\max_{1\leq k\leq \fl{Ns}}\Big\{\sum_{i=1}^kI^{(u)}_{i,0}-\sum_{j=1}^{\fl{Nt}}J^{(u)}_{0,j}+G_{(k,1),(\fl{Ns},\fl{Nt})}\Big\}\\
&\hspace{2cm}\bigvee\max_{1\leq k\leq \fl{Nt}}\Big\{-\sum_{j=k+1}^{\fl{Nt}}J^{(u)}_{0,j}+ \om_{1, k} +G_{(1,k),(\fl{Ns},\fl{Nt})}\Big\}.
\end{align*}
To compactify notation we use a convention where the $y$-axis is labeled by negative indices and we define 
\be\label{etak}
\text{for }k\in\Z\qquad \eta_k=
\begin{cases}
-\sum_{j=-k+1}^{\fl{Nt}}J^{(u)}_{0,j}\qquad&k\leq0,\\
\sum_{i=1}^kI^{(u)}_{i,0}-\sum_{j=1}^{\fl{Nt}}J^{(u)}_{0,j}\qquad&k\geq1.
\end{cases}
\ee
As such, we can say that the last passage time can be obtained on path that enters the bulk $\N^2$ at a point $\mathbf{v}(z)$ defined by 
\be\label{vz}
\text{for }z\in\R\qquad \mathbf{v}(z)=
\begin{cases}
(1,\fl{-z})\qquad&z\leq-1,\\
(1,1)\qquad&-1<z<1,\\
(\fl{z},1)\qquad&z\geq1,\\
\end{cases}
\ee
and the gradient can be written as 
\[
\sum_{i=1}^{\fl{Ns}}I^{(u)}_{i,\fl{Nt}}=\max_{\fl{-Nt}\leq k\leq\fl{Ns},k\not=0}\big\{\eta_k+ \om_{{\bf v}(k)}\mathbbm 1\{ k<0 \}+ G_{{\bf v}(k),(\fl{Ns},\fl{Nt})}\big\}.
\]
Then the following inequalities are immediate:
\begin{align}\label{ineq1}
\eta_k&+G_{{\bf v}(k),(\fl{Ns},\fl{Nt})}\\
&\hspace{1cm}\leq \sum_{i=1}^{\fl{Ns}}I^{(u)}_{i,\fl{Nt}} \notag\\
&\hspace{2cm}\leq \max_{\fl{-Nt}\leq k\leq\fl{Ns},k\not=0}\{\eta_k+G_{{\bf v}(k),(\fl{Ns},\fl{Nt})}\}+1.\label{ineq2}
\end{align}
This inequality will be crucial for our purposes. We briefly discuss the main idea. 

The second line in the last display is a sum of i.i.d. Bernoulli, so it has a known large deviation rate function. A deviation for the $\sum I^{(u)}$ is controlled above and below by deviations for the expressions $ \eta_k +G_{{\bf v}(k),(\fl{Ns},\fl{Nt})}$. $\eta_k$ itself is either a sum of i.i.d.\ geometric random variables or a difference of two independent sums; in either case the large deviation rate function for $\eta_k$ is computable, and the only unknown will be the large deviation rate function for $G$ (albeit in a complicated expression). The subsection is devoted into following this program and solve for the rate function of  $G$.

It will be crucial to understand the function defined by 
\be\label{eq:Hst} 
H_{s,t}^{a,b}(r)=-\lim_{N\to\infty}N^{-1}\log\P\{\eta_{\fl{Na}}+G_{\mathbf{v}(Nb),(\fl{Ns},\fl{Nt})}\geq Nr\},
\ee
where $a, b \in [-t,s]$. 
We first argue why the limit exists. This fact will be a direct consequence of Lemma \ref{srf}, when we show that the $\eta_{\fl{Na}}$ and  $G_{\mathbf{v}(Nb),(\fl{Ns},\fl{Nt})}$ will have a right tail rate function. 

We begin by computing the rate function for the $\eta_k$.
%For the i.i.d. weights $\{I_{i,\fl{Nt}}\}$ we have the right-tail Cramer rate function
%\begin{align}
%R_s(r)&=-\lim_{N\to\infty}N^{-1}\log\P\Big\{N^{-1}\sum_{i=1}^{\fl{Ns}}I_{i,\fl{Nt}}\geq N r\Big\}\label{Rsr}\\
%&=
%\begin{cases}
%s I_{u}(r/s),\qquad &s \E[I_{1,0}]\leq r\leq s,\\
%0,\qquad&0 \leq r\leq s \E[I_{1,0}],
%\end{cases}\notag
%\end{align}
%where the rate function $I_{u}(r)$ is
%\be
% I_\mathcal B^{(u)}(r)=r\log \frac{r}{u}+(1-r)\log\frac{1-r}{1-u}\qquad \text{with }r\in[0,1].
%\ee
%The convex dual of $R_s$ is given by
%\be\label{Rstar}
%R_s^*(\xi)=sC^{(u)}_\mathcal B(\xi),\qquad \text{with }\xi\in \R
%\ee
%where $C^{(u)}_\mathcal B(\xi)$ is the Bernoulli logarithm generating function we parameter $u\in(p,1]$ in \eqref{cumB}.
For real $a\in[-t,s]$, and $r \in \R$ define 
\be
\kappa_a(r)=-\lim_{N\to\infty}N^{-1}\log\P\{\eta_{\fl{Na}}\geq Nr\}.
\ee
From \eqref{etak} we observe that if $k\leq0$ $\eta_k$ is a sum of i.i.d.\ geometric distributed random variables while if $k\geq1$, $\eta_k$ is the difference of two independent sums of i.i.d.\ random variables. 

%We first analyze the case $k\leq0$, where $\eta_k$ is a sum of i.i.d.\ Geometric$(\frac{u-p)}{u(1-p)})$. Then, for 
%$-t\leq a \leq 0$
%\[
%\kappa_a(r)=
%\begin{cases}
%(t+a)I^{(\frac{u-p)}{u(1-p)})}_\mathcal G(-r),\qquad &r<-t \E[J_{0,1}]^{-1},\\
%0,\qquad &-t \E[J_{0,1}]^{-1}\leq r<0,\\
%\infty,\qquad &r\geq 0.
%\end{cases}
%\]
%For the case $k\geq1$, we use Lemma \ref{srf} accordingly defining $L_N=\sum_{i=1}^kI_{i,0} $ and $Z_N=-\sum_{j=1}^{\fl{Nt}}J_{0,j}$
%\[
%\kappa_a(r)=\inf_{au\leq z\leq a\wedge r-t\frac{u-p}{p(1-u)}}\Big\{tI^{(\frac{u-p}{u(1-p)})}_\mathcal G(z-r)+a I_\mathcal B^{(u)}(z/a)\Big\}.
%\]
The convex dual is
\be \label{kstar}
\begin{aligned}
\kappa^*_a(\xi)&=\sup_{r\in\R}\{\xi r-\kappa_a(r)\}\\
&=\begin{cases}
(t+a)\Big[\log\frac{u-p}{u(1-p)}-\log\Big(1-\frac{p(1-u)}{u(1-p)}e^{-\xi}\Big)\Big], \qquad\text{for }\xi>\log\frac{p(1-u)}{u(1-p)},-t\leq a \leq0,\\
t\Big[\log\frac{u-p}{u(1-p)}-\log\Big(1-\frac{p(1-u)}{u(1-p)}e^{-\xi}\Big)\Big]+a\log(ue^\xi+1-u), \\
\hspace{7.8cm}\text{for }\xi>\log\frac{p(1-u)}{u(1-p)},\,\,0< a \leq s,\\
\infty,  \hspace{7.3cm}\text{otherwise}.
\end{cases}\\
&=\begin{cases}
(t+a) C _{\mathcal G}^{(\frac{u - p}{u(1-p)})}(-\xi), &\text{for }\xi>\log\frac{p(1-u)}{u(1-p)},-t\leq a \leq0,\\
t C _{\mathcal G}^{(\frac{u - p}{u(1-p)})}(-\xi)+a C _{\mathcal B}^{(u)}(\xi), &\text{for }\xi>\log\frac{p(1-u)}{u(1-p)},\,\,0< a \leq s,\\
\infty,  &\text{otherwise}.
\end{cases}
\end{aligned}
\ee
The first line in \eqref{kstar} follows from Cramer's theorem when the random variables are geometric. The second line follows from Lemma \ref{srf} when 
$L_N =\sum_{i=1}^{\fl{Na}}I^{(u)}_{i,0} $ and $Z_N = -\sum_{j=1}^{\fl{Nt}}J^{(u)}_{0,j}$, and the fact that the dual of an infimal convolution is the sum of the corresponding duals.

\begin{remark} 
\label{rem:important}
The condition on $\xi$ can be stated equivalently in terms of $u$. In fact, if $\xi\in\R$ is fixed, the above inequality becomes $u>\frac{pe^{-\xi}}{1-p+pe^{-\xi}}$. Moreover if $\xi>0$, $\frac{pe^{-\xi}}{1-p+pe^{-\xi}}<p$ and so it remains $u\in(p,1]$. 
%While if $\xi<0$, $\frac{pe^{-\xi}}{1-p+pe^{-\xi}}\in(p,1]$, which implies that the condition on $u$ becomes $u\in(\frac{pe^{-\xi}}{1-p+pe^{-\xi}},1]$. This will play a role in the sequence.
\end{remark}

The rightmost zero $m_{\kappa,a}$ of $\kappa_a$ is the law of large numbers limit
\be\label{mka}
m_{\kappa,a}=\lim_{N\to\infty}N^{-1} \eta_{\fl{Na}}=
\begin{cases}
-(t+a)\frac{u-p}{p(1-u)},\qquad&-t\leq a\leq 0,\\
au-t\frac{u-p}{p(1-u)},\qquad&0< a\leq s.
\end{cases}
\ee
Note that when viewed as functions of $a$,  $\kappa_a$, $\kappa_a^*$ and $m_{\kappa,a}$ are all continuous at $a=0$.

For the rate function of $G_{\mathbf{v}(Nb),(\fl{Ns},\fl{Nt})}$, we first introduce the equivalent macroscopic version of \eqref{vz} for 
$a  \in \R$, by 
\be
N^{-1}\mathbf{v}(Na)\to\bar{\mathbf{v}}(a)=
\begin{cases}
(0,-a),\qquad&-t\leq a \leq 0,\\
(a,0),\qquad& 0<a\leq s.
\end{cases}
\ee
With this notation, the rate function of the last past passage time in the interior is
\be\label{Jva}
J_{(s,t)-\bar{ \mathbf{v}}(a)}(r)=-\lim_{N\to\infty}N^{-1}\log\P\{G_{\mathbf{v}(Na),(\fl{Ns},\fl{Nt})}\geq Nr\}.
\ee
This is because $G_{\mathbf{v}(Na),(\fl{Ns},\fl{Nt})}$ equals in distribution $G_{(0,0), (\fl{Ns},\fl{Nt}) - \mathbf{v}(Nb)}$. There will be a small discrepancy between 
 $(\fl{Ns},\fl{Nt})-\mathbf{v}(Na)$ and $\fl{N ((s,t)-\bar{ \mathbf{v}}(a))}$ but  Lemma \ref{lem:1} proves that it is negligible in the limit.

Let $m_{\kappa,a}$ and $m_{J,b}$ be the rightmost zeros respectively of $\kappa_a$ (defined by  \eqref{mka}) and $J_{(s,t)-\bar{\mathbf{v}}(b)}$ 
(which equals $g_{pp}((s,t) - \bar{\mathbf{v}}(b))$. Using Lemma \ref{srf} for $(a,b)\in[-t,s]^2$, we have
\be\label{Hdef}
 H_{s,t}^{a,b}(r)=
\begin{cases}
0,\qquad &r<m_{\kappa,a}+m_{J,b},\\
\inf_{m_{\kappa,a}\leq x\leq r-m_{J,b}}\{\kappa_a(x)+J_{(s,t)-\bar{\mathbf{v}}(b)}(r-x)\},\qquad &m_{\kappa,a}+m_{J,b} \leq r\leq s.
\end{cases}
\ee
The following regularity lemma follows from the continuity properties we discussed up to this point, and the details are left to the reader.
\begin{lemma}\label{reg}
Fix $s,t\in(0,\infty)$ and fix any compact set $K \subseteq(-\infty,s]$. Then $ H_{s,t}^{a,b}(r)$ is a uniformly continuous function of $(b,r)\in[-t,s]\times K$, uniformly in $a\in[-t,s]$. In symbols
\be
\lim_{\delta \to 0}\sup_{\substack{a,b,b'\in[-t,s],r,r'\in K:\\|b-b'|\leq\delta,|r-r'|\leq\delta}}|H^{a,b}_{s,t}(r)-H^{a,b'}_{s,t}(r')|=0.
\ee
\end{lemma}

When $a=b$ we simplify the notation as $H^a_{s,t}(r)=H^{a,a}_{s,t}(r)$. Observe that at this point an expression involving  $J_{s,t}$ manifested on the right-hand side of  \eqref{Hdef}. Our goal is to invert the relation and isolate $J_{s,t}$. 

The next lemma is the continuous version of the discrete inequalities \eqref{ineq1}, \eqref{ineq2} at the level of the rate functions. 
\begin{lemma}\label{Rlem}
Let $s,t\in(0,\infty)$ and $r\in[0,s]$. Then
\be\label{Req}
sI_{\mathcal B}^{(u)}(r/s)=\inf_{-t\leq a\leq s}H^a_{s,t}(r).
\ee
\end{lemma}
\begin{proof}
For any $a\in[-t,s]$, by \eqref{ineq1}
\begin{align*}
-sI_{\mathcal B}^{(u)}(r/s)&=\lim_{N\to\infty}N^{-1}\log\P\Big\{\sum_{i=1}^{\fl{Ns}}I^{(u)}_{i,\fl{Nt}}\geq Nr\Big\}\\
&\geq \lim_{N\to\infty}N^{-1}\log\P\{\eta_{\fl{Na}}+G_{\mathbf{v}(\fl{Na}),(\fl{Ns},\fl{Nt})}\geq Nr\}\\
&=-H^a_{s,t}(r).
\end{align*}
This is true for an arbitrary $a$, therefore
\be
sI_{\mathcal B}^{(u)}(r/s) \leq \inf_{-t\leq a\leq s}H^a_{s,t}(r).
\ee

To get the lower bound we use \eqref{ineq2} together with a coarse graining argument.

We begin describing the partition which will be helpful when we will use  \eqref{ineq2}. Fix a small enough $\delta>0$ to partition the interval $[-t,s]$. In particular, define 
 $-t=a_0<a_1<\dots<a_q=0<\dots<a_m=s$ where $|a_{i+1}-a_i|<\delta$. Moreover, we fix an $\e > 0$ and we assume that $N$ is large enough so that $N\e >1$.

When  $a_i\geq0$, for any  $k\in[\fl{Na_i},\fl{Na_{i+1}}]\cap \Z$,
\begin{align*}
\P\{\eta_{k}+G_{\mathbf{v}(k),(\fl{Ns},\fl{Nt})}\geq Nr\} \leq\P\Big\{ \eta_{\fl{Na_{i+1}}}+G_{\mathbf{v}(\fl{Na_{i}}),(\fl{Ns},\fl{Nt})}\geq Nr\Big\}. 
%&\leq\P\{\eta_{\fl{Na_{i+1}}}+G_{\mathbf{v}(\fl{Na_{i+1}}),(\fl{Ns},\fl{Nt})}\geq N(r-\e)\}+\P\Big\{\sum_{i=\fl{Na_{i}}+1}^{\fl{k}}I_{i,0}+\sum_{j=\fl{k}}^{\fl{Na_{i+1}}-1}\om_{j,1}\geq N\e\Big\}. 
\end{align*}
Similarly, when $a_i<0$ and $\fl{Na_{i}}<k \leq\fl{Na_{i+1}}$ the bound becomes
\begin{align*}
\P\{\eta_{ k}+G_{\mathbf{v}(k),(\fl{Ns},\fl{Nt})}\geq Nr\}\leq\P\Big\{ \eta_{\fl{Na_{i}}}+G_{\mathbf{v}(\fl{Na_{i+1}}),(\fl{Ns},\fl{Nt})}\geq Nr\Big\}.
\end{align*}
From \eqref{ineq2} we bound 
\begin{align*}
 \P\Big\{\sum_{i=1}^{\fl{Ns}}I_{i,\fl{Nt}}\geq Nr\Big\}&\leq \P\Big\{\max_{\substack{\fl{-Nt}\leq k\leq\fl{Ns},\\k\not=0}}\{\eta_k+G_{\mathbf{v}(k),(\fl{Ns},\fl{Nt})}\} +1 \geq Nr\Big\}\\
&\leq \P\Big\{\max_{\substack{\fl{-Nt}\leq k\leq\fl{Ns},\\k\not=0}}\{\eta_k+G_{\mathbf{v}(k),(\fl{Ns},\fl{Nt})}\} \geq N(r-\e)\Big\}.
\end{align*}
Take logarithm on both sides and divide by $N$ and use a union bound to obtain
\begin{align*}
N^{-1}&\log \P\Big\{\sum_{i=1}^{\fl{Ns}}I_{i,\fl{Nt}}\geq Nr\Big\} \\
&\leq N^{-1}\log m + \Big\{\max_{\substack{0\leq i\leq q-1}}\Big\{N^{-1}\log\P\{\eta_{\fl{Na_{i}}}+G_{\mathbf{v}(\fl{Na_{i+1}}),(\fl{Ns},\fl{Nt})}\geq N(r-\e)\}\Big\} \Big\}\\
&\hspace{2.3cm}\vee \Big\{\max_{\substack{q\leq i\leq m-1}}\Big\{N^{-1}\log\P\{\{\eta_{\fl{Na_{i+1}}}+G_{\mathbf{v}(\fl{Na_{i}}),(\fl{Ns},\fl{Nt})}\geq N(r-\e)\}\Big\}.
\end{align*}
Take $N\to\infty$ to get
\begin{align*}
-sI_{\mathcal B}^{(u)}(r/s) &\leq \Big\{\max_{\substack{0\leq i\leq q-1}}\{-H_{s,t}^{a_{i},a_{i+1}}(r-\e)\}\Big\}\vee \Big\{\max_{\substack{q\leq i\leq m-1}}\{-H_{s,t}^{a_{i+1},a_i}(r-\e)\}\Big\}\\
&\leq\sup_{a,b\in[-t,s]:|a-b|\leq\delta}\{-H_{s,t}^{a,b}(r-\e)\}.
\end{align*}
Use Lemma \ref{reg} by letting $\delta \to 0$; this also implies $b \to a$. Then let $\e \to 0$. 
\end{proof}

The following lemma is the last technical tool we need in order to finally solve \eqref{Req} for the unknown rate function $J$. It proves convexity and lower semi-continuity of the Legendre dual of $J$. 

\begin{lemma}\label{GJ}
For a fixed $\xi\in\R_+$, the function $J^*_{s,t}(\xi)$, as a function of $(s,t)$,
is continuous and finite on $\R^2_+$.
\end{lemma}

%\begin{remark}[Restriction to $\xi \ge 0$]  Function $M_{\xi}$ will be central to the calculations that follow, so we do need to make sure it takes some finite values for the equations to be meaningful. As it happens, when $\xi < 0$,  $M_{\xi}(a) = -\infty$ for $a \in [-t, s]$ and $M_{\xi}$ takes no finite values. Therefore, from this point onwards, we restrict to $\xi \ge 0$.
%\end{remark}

\begin{proof}
By definition 
$J^*_{s,t}(\xi)= \sup_{r\in \R}\{\xi r-J_{s,t}(r)\},$ 
but, since $J_{s,t}(r) = \infty$ for $r > s$, and  $J_{s,t}(r) = 0$ for $r < g_{pp}(s,t)$, we can write for $\xi \ge 0$ that 
\[ J^*_{s,t}(\xi)= \sup_{r\in [g_{pp}(s,t), s]}\{\xi r-J_{s,t}(r)\}. \]
Then it is immediate to see that 
\[ J^*_{s,t}(\xi) \le \xi s, \quad \text{for all} (s,t) \in \R^2_+.\]

Continuity will follow once we prove that $M_\xi(s,t)$ is a concave finite function. Let $\lambda\in(0,1)$ and $(s,t)=\lambda(s_1, t_1)+(1-\lambda)(s_2, t_2)$ 
for some $(s_i, t_i) \in \R^2_+$. 
Recall that $J$ is convex and lower-semicontinuous in $(s,t,r)$ from Theorem \ref{JInt}. Write $r$ as the convex combination  $r=\lambda r_1+(1-\lambda)r_2$ for some  $r_1, r_2\in \R$.  By convexity 
 \begin{align*}
&\inf_{r\in \R}\{J_{s,t}(r)-\xi r\}\\
&\hspace{0.2cm}\leq\inf_{r\in \R}\Big\{\inf_{\substack{(r_1,r_2): \lambda r_1+(1-\lambda)r_2=r}}\{\lambda(J_{s_1,t_1}(r_1)-\xi r_1)+(1-\lambda)(J_{s_2,t_2}(r_2)-\xi r_2)\}\Big\}\\
&\hspace{0.2cm}=\inf_{\substack{(r_1,r_2)\in \R^2}}\{\lambda(J_{s_1,t_1}(r_1)-\xi r_1)+(1-\lambda)(J_{s_2,t_2}(r_2)-\xi r_2)\}\\
&\hspace{0.2cm}=\lambda\inf_{\substack{r_1\in \R]}}\{J_{s_1,t_1}(r_1)-\xi r_1\}+(1-\lambda)\inf_{\substack{r_2\in \R}}\{J_{s_2,t_2}(r_2)-\xi r_2\}\\
&\hspace{0.2cm}=-\lambda J^*_{s_1,t_1}(\xi)-(1-\lambda)J^*_{s_2,t_2}(\xi).
\end{align*}
In the end we have
\[J^*_{s,t}(\xi)\geq \lambda J^*_{s_1,t_1}(\xi)+(1-\lambda)J^*_{s_2,t_2}(\xi),\]
which is enough to prove the concavity of $J^*_{s,t}(\xi)$ in $(s,t)$.  
\end{proof}

\begin{figure}
\begin{center}
\includegraphics[height=4.1cm]{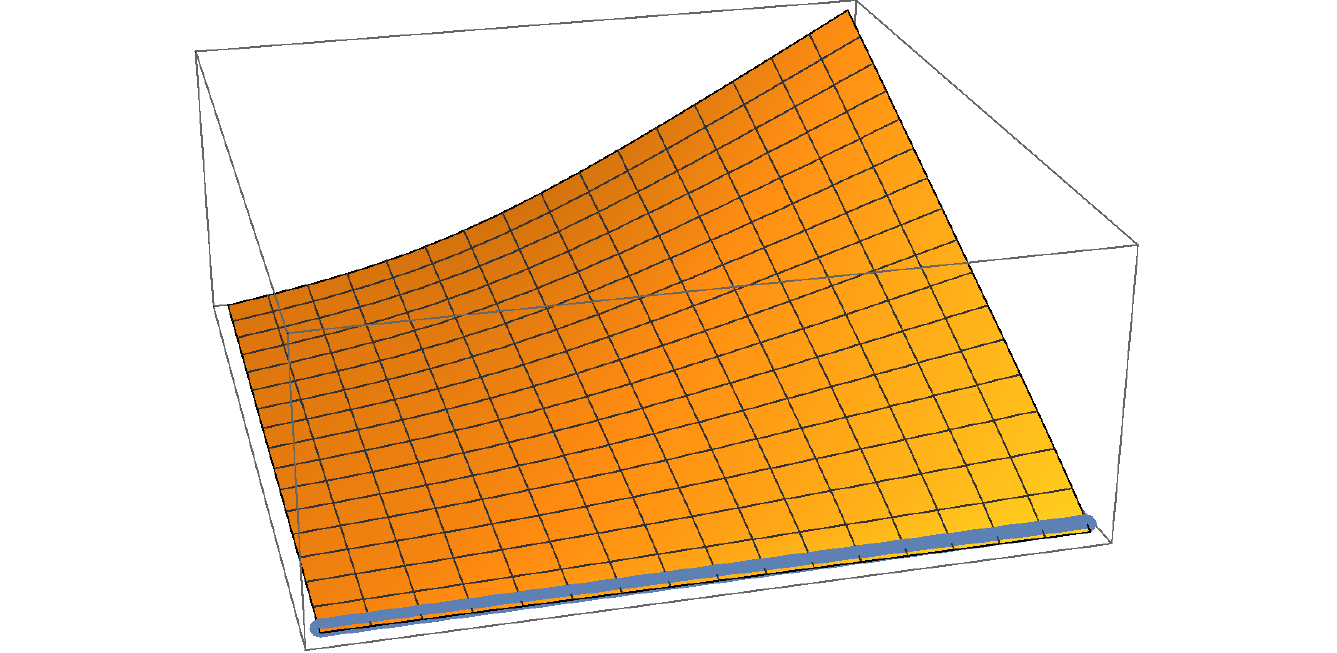}
\hspace{0cm}
\includegraphics[height=3.4cm]{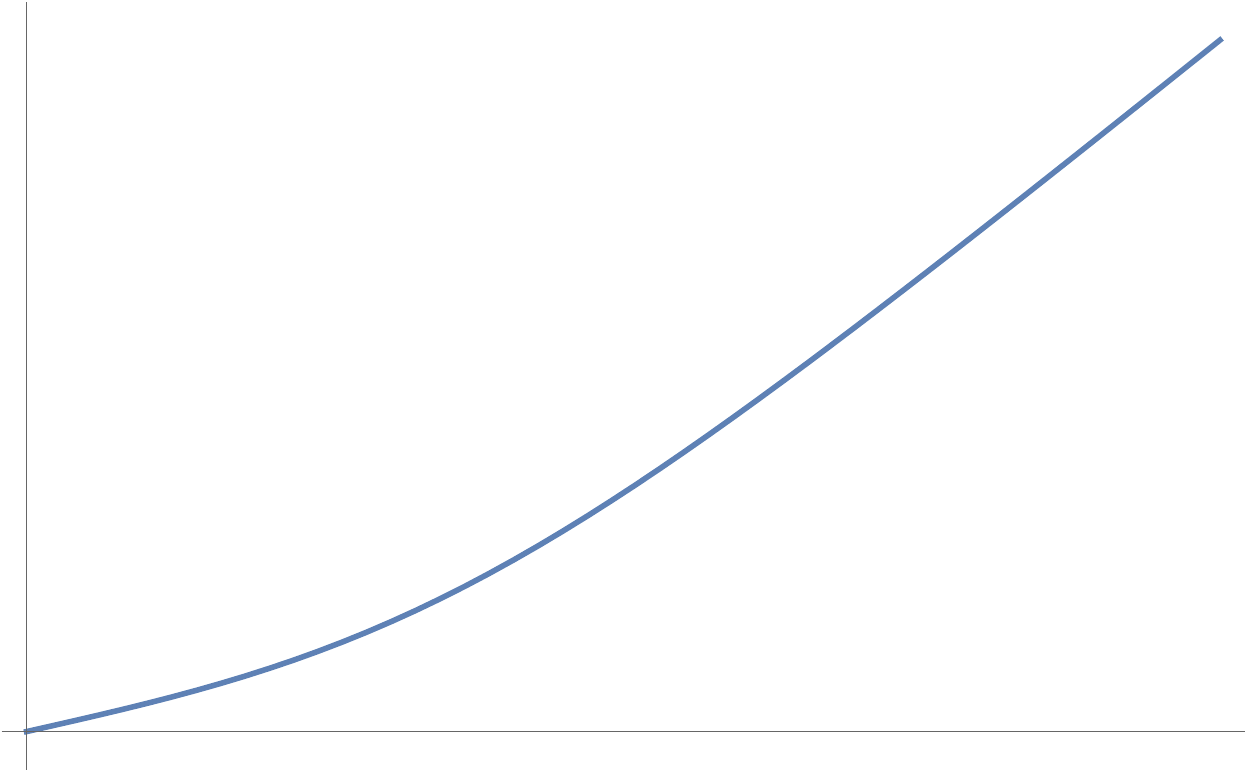}
\end{center}
\caption{Graphical representation of the function $J^*_{s,t}(r)$. In both figures we used $p=0.1$ and $t=1$. To the left we have $J^*_{s,1}(\xi)$ as a function of $(s,\xi)$ and one see the directions of convexity when $s$ is fixed and $\xi$ varies, and the direction of concavity  ranges when $\xi$ is fixed and $s$ varies as described in the proof of Lemma \ref{GJ}. The blue line is at $s = 1/9$ which the is characteristic point for  $p=0.1$ and $t=1$. For smaller $s$, $J^*_{s,1}(\xi) = s\xi$.  To the right is the convex continuous function $J_{10,1}^*(\xi)$.}
\end{figure}

In order to prove Proposition \ref{propJ} we need the following technical result. This is in the spirit of Proposition 3.10 in \cite{Ja-15} but tailored to our particular case. For this reason we postpone the proof until  Appendix \ref{app:conv}.

\begin{proposition}\label{prop:dousup}
Let $I=(a,b]\subseteq \R$ with $a,b\in \R$. Let the convex functions $h,g$: $I\to \R$ be twice continuously differentiable with $h'(u)>0$ and $g'(u)<0$ for every $u\in I$. Define 
\[
f_{s,t}(u)=sh(u)+tg(u)\quad \text{with }(s,t)\in \R_+^2.
\]
Suppose that $f''_{s,t}(u)>0$ for all $(s,t)\in \R_+^2$, $\lim_{u\searrow a}f_{s,t}(u)=\infty$ and $f_{s,t}(b)=c<\infty$ with $c\in \R$. If $\Lambda(s,t)$ is a continuous function in $(s,t)$ with the property that for all $(s,t)\in \R_+^2$ and $u\in I$ the identity
\be\label{ii}
0=\sup_{0\leq z \leq s}\{\Lambda(s-z,t)-f_{s-z,t}(u)\}\vee\sup_{0\leq \tilde{z} \leq t}\{\Lambda(s,t-\tilde{z})-f_{s,t-\tilde{z}}(u)\}
\ee
holds, then for every $ t<-\frac{h'(b)}{g'(b)}s$,
\[
\Lambda(s,t)=\min_{u\in I}\{f_{s,t}(u)\}.
\]
\end{proposition}

Using this we can now find a variational expression for $J^*$.

\begin{proof}[Proof of Proposition \ref{propJ}]

If $\xi<0$, by definition 
\[
J^*_{s,t}(\xi)=\sup_{r\in \R}\{r\xi-J_{s,t}(\xi)\}=\sup_{r< g_{pp}(s,t)}\{r\xi-J_{s,t}(r)\}\vee\sup_{r\in[g_{pp}(s,t),s]}\{r\xi-J_{s,t}(r)\}\vee\sup_{r>s}\{r\xi-J_{s,t}(r)\}.
\]
Note that the first supremum is $+\infty$ since $J_{s,t}(r)=0$ for $r< g_{pp}(s,t)$ and $\xi<0$. Therefore $J^*_{s,t}(\xi)=\infty$ if $\xi<0$.

If  $\xi \ge 0$, equation \eqref{Hdef} gives that $H^a_{(s,t)}$ is the infimal convolution  of $\kappa_a$ and $J_{(s,t)-\bar{\mathbf{v}}(a)}$ since the value of the infimum does not change when we allow $x$ to range over all of $\R$.  
We compactify the notation by writing  $H^a_{s,t}(r)=\kappa_a\Box J_{(s,t)-\bar{\mathbf{v}}(a)}(r)$. By Theorem 16.4 in \cite{Roc-70}, the addition operation is dual to the infimal convolution operation. From \eqref{Req} of Lemma \ref{Rlem}, take the Legendre dual on both sides to obtain 
\be\label{Rstara}
\begin{aligned}
s C_{\mathcal B}^{(u)}(\xi)&=\sup_{-t\leq a\leq s}\Big\{\sup_{r\in \R}\{r\xi-(\kappa_a\Box J_{(s,t)-\bar{\mathbf{v}}(a)})(r)\}\Big\}\\
&=\sup_{-t\leq a\leq s}\big\{(\kappa_a\Box J_{(s,t)-\bar{\mathbf{v}}(a)})^*(\xi)\big\}=\sup_{-t\leq a\leq s}\big\{(\kappa_a^*(\xi)+ J^*_{(s,t)-\bar{\mathbf{v}}(a)}(\xi)\big\}.
\end{aligned}
\ee
%Substitute in \eqref{eq:bercum},to obtain for any fixed $\xi \ge 0$
%\[ s\log(ue^\xi+1-u)=\sup_{-t\leq a\leq s}\big\{(\kappa_a^*(\xi)+ J^*_{(s,t)-\bar{\mathbf{v}}(a)}(\xi)\big\} .\]
From \eqref{kstar} we can substitute the explicit expression of $\kappa_a^*(\xi)$. 
Define 
\be\label{eq:hd}
\begin{aligned}
-\ell_\xi(u)&=C_\mathcal G^{(\frac{u-p}{u(1-p)})}(-\xi)=\log\frac{u-p}{u(1-p)-p(1-u)e^{-\xi}},\\
d_\xi(u)&=C_\mathcal B^{(u)}(\xi)=\log(ue^\xi+1-u).
\end{aligned}
\ee
Use this to simplify \eqref{Rstara} into 
\[
 sd_\xi(u)+t\ell_\xi(u)=\sup_{0\leq a\leq t}\{a\ell_\xi(u)+ J^*_{s,t-a}(\xi)\big\}\bigvee \sup_{0\leq a\leq s}\{ad_\xi(u) + J^*_{s-a,t}(\xi)\big\}.
\]
Subtract $sd_\xi(u)+t\ell
_\xi(u)$ to either side
\[
0= \sup_{0\leq z\leq s}\{J^*_{s-a,t}(\xi)-[(s-a)d_\xi(u)+t\ell_\xi(u)]\big\}\vee \sup_{0\leq \tilde{z}\leq t}\{J^*_{s,t-a}(\xi)-[sd_\xi(u)+(t-a)\ell_\xi(u)]\big\}.
\]
Use Proposition \ref{prop:dousup} identifying as $I=(p,1]$, $\Lambda(s,t)=J^*_{s,t}(\xi)$, $h(u)=d_\xi(u)$, $g(u)=\ell_\xi(u)$ and therefore $f_{s,t}(u)= sd_\xi(u)+t\ell_\xi(u)$. The only hypothesis that is not immediately verifiable is continuity of $J^*$ in $s, t$, but that is now covered by Lemma \ref{GJ}.
%Note that
%\begin{align*}
%\ell'_\xi(u)=g'(u)&=\frac{u-p}{u(1-p)-p(1-u)e^{-\xi}}\frac{(1-p+pe^{-\xi})(u-p)-(u(1-p)-p(1-u)e^{-\xi})}{(u-p)^2}\\
%&=\frac{p(p-1)(1-e^{-\xi})}{(u-p)(u(1-p)-p(1-u)e^{-\xi})}
%<0\qquad\text{for } u\in(p,1].
%\end{align*}
%and 
%\[
%d'_\xi(u)=h'(u)=\frac{e^\xi-1}{ue^\xi+1-u}>0\qquad\text{if }u\in(p,1].
%\]
%Moreover $f''_{s,t}(u)>0$ for every $(s,t)\in \R^2_+$.  And $\lim_{u\searrow p}f_{s,t}(u)=\infty$ and $f_{s,t}(1)=s\xi <\infty$. 
Therefore, if $t<\frac{1-p}{p}s$
\[
J^*_{s,t}(\xi)=\min_{u\in(p,1]}\{ sd_\xi(u)+t\ell_\xi(u)\}=\min_{u\in(p,1]}\{ sC_\mathcal B^{(u)}(\xi)-tC_\mathcal G^{(\frac{u-p}{u(1-p)})}(-\xi)\}.
\]
For $t \ge \frac{1-p}{p}s$ we reason directly: $J_{s,t}(r) = + \infty \mathbbm 1\{ r > s\}$ and its convex dual will be $s \xi$ for $\xi > 0$. This is also 
the $\min_{u\in(p,1]}\{ sC_\mathcal B^{(u)}(\xi)-tC_\mathcal G^{(\frac{u-p}{u(1-p)})}(-\xi)\}$, with the minimum obtained at $u = 1$.
\end{proof}

\subsection{Closed formula for $J^*_{s,t}(\xi)$}

\begin{proof}[Proof of Theorem \ref{thm:lam}]
The aim of this proof is to find an analytical result for the infimum in Proposition \ref{propJ} when $t<\frac{1-p}{p}s$. Therefore we start computing the derivatives of the two cumulant-generating function and to find the optimizing point we solve the equation
\begin{align*}
0=&s\frac{\partial C^{(u)}_\mathcal B(\xi)}{\partial u}-t\frac{\partial C^{(\frac{u-p}{u(1-p)})}_\mathcal G(-\xi)}{\partial u}=s\frac{e^\xi-1}{1+u(e^\xi-1)}\\
&\hspace{3.5cm}-t\frac{p(p-1)(e^{-\xi}-1)}{u^2(1+p(e^{-\xi}-1))-up[1+e^{-\xi}+p(e^{-\xi}-1)]+p^2e^{-\xi}}
\end{align*}
or equivalently, after the algebraic simplification of denominators
\begin{align*}
0=&u^22s[(1-p)(e^\xi-1)+p(1-e^{-\xi})]-up[(1-p)(s+t)(e^\xi+e^{-\xi}-2)+2s(1-e^{-\xi})]\\
&\phantom{xxxx}+(e^{-\xi}-1)p((1-p)(s+t)-s).
\end{align*}
The minimum is in fact attained to the solution to this equation (for further details see Appendix \ref{sec:app1}). The minimizing point is
\be\label{eq:umin}
u^*=\frac{p(1-p)(s+t)(e^\xi+e^{-\xi}-2)+2sp(1-e^{-\xi})+\sqrt{\Delta}}{2s[(1-p)(e^\xi-1)+p(1-e^{-\xi})]}
\ee
with $\Delta=p(1-p)(e^\xi+e^{-\xi}-2)[(1-p)p(s+t)^2(e^\xi+e^{-\xi}-2)+4st]$.
Then \eqref{exactform} follows directly by 
\[
J_{s,t}^*(\xi) = s C_{\mathcal B}^{(u^*)}(\xi) - t s C_{\mathcal G}^{(u^*)}(-\xi). \qedhere
\]

\end{proof}

\section{Invariant model: Limiting log-moment generating functions}
\label{sec:immgf}

Define the last passage time's l.m.g.f. for the boundary model
\be\label{cmgfb}
\Lambda^{(u)}_{(s,t)}(\xi)=\lim_{N\to\infty} N^{-1}\log\E e^{\xi G^{(u)}_{\fl{Ns},\fl{Nt}}}.
\ee
In this section we find $\Lambda^{(u)}_{(s,t)}(\xi)$ when $\xi >0$. 

It will be convenient to also define the l.m.g.f.\ for the two passage times conditional on the first step 
being $e_1$ or $e_2$,  $G^{(u),\text{hor}}_{\fl{Ns},\fl{Nt}}$ and $G^{(u),\text{hor}}_{\fl{Ns},\fl{Nt}}$ 
given by \eqref{lpthor} and \eqref{lptver} respectively. The corresponding  l.m.g.f.\ are
\be\label{lamhor}
\Lambda^{(u),\text{hor}}_{(s,t)}(\xi)=\lim_{N\to\infty}N^{-1}\log\E e^{\xi G^{(u),\text{hor}}_{\fl{Ns},\fl{Nt}}}
\ee
and
\be\label{lamver}
\Lambda^{(u),\text{ver}}_{(s,t)}(\xi)=\lim_{N\to\infty}N^{-1}\log\E e^{\xi G^{(u),\text{ver}}_{\fl{Ns},\fl{Nt}}}.
\ee
The existence of the above limits is verified in Lemma \ref{justif} below, but we state it as part of the main Theorem \ref{thm:boundary}.

The existence of the two limits above then gives rise to the formula 
\be\label{lambound}
\Lambda^{(u)}_{(s,t)}(\xi)=\Lambda^{(u),\text{hor}}_{(s,t)}(\xi)\vee\Lambda^{(u),\text{ver}}_{(s,t)}(\xi) \quad \text{for any } \xi >0.
\ee
Thus, finding $\Lambda^{(u)}_{(s,t)}(\xi)$ is equivalent to finding $\Lambda^{(u),\text{hor}}_{(s,t)}(\xi), \Lambda^{(u),\text{ver}}_{(s,t)}(\xi)$,
which is the content of Theorem \ref{thm:boundary} below.

Heuristically, one expects the creation of some critical direction for $(s,t)$ that will depend on $\xi, p, u$; 
below the direction the boundary effect will be felt at the l.g.m.f.\ level, and otherwise the model will behave like the boundary is not present. This was also observed at the LLN level in Theorem \ref{thm:HOR}. In fact this is the case.

For $\xi > 0$ we define
\be \label{eq:ku}
k^{(u)}(\xi)=\Big(\frac{\partial C_\mathcal B^{(u)}(\xi)}{\partial u}\Big)/\Big(\frac{\partial C_\mathcal G^{(u)}(-\xi)}{\partial u}\Big).
\ee
The relevant conditions that create a critical line are
\be\label{lmgft}
t = k^{(u)}(\xi) s, \quad \text{ and } t = k^{(u)}(-\xi) s,
\ee
for $\Lambda^{(u), \text{hor}}$ and $\Lambda^{(u), \text{ver}}$ respectively. Recall that $l.m.g.f$ of $G_{\fl{Ns},\fl{Nt}}$ is given by Corollary \ref{cor:lam}, and is equal to $I^*_{s,t}(\xi) = J^*_{s,t}(\xi)$. For uniformity of notation in the section, set  $\Lambda_{(s,t)}(\xi) = I^*_{s,t}(\xi)$.

\begin{theorem}\label{thm:boundary}
Let $s,t\geq0$, $u\in(p,1)$ and $\xi\geq0$. 
\begin{enumerate}[(a)]
	\item The limit in \eqref{lamhor} exists and is given by
		\be\label{lamhorc}
			\Lambda^{(u),{ \rm hor}}_{(s,t)}(\xi)=
				\begin{cases}
					sC_\mathcal B^{(u)}(\xi)-tC_\mathcal G^{(\frac{u-p}{u(1-p)})}(-\xi)\qquad&\text{if }
					t< k^{(u)}(\xi)s,\\
					\Lambda_{(s,t)}(\xi)\qquad&\text{if }t\geq k^{(u)}(\xi) s.
				\end{cases}
		\ee
	\item The limit in \eqref{lamver} exists and is given by
		\be \label{lamverc}
			\Lambda^{(u),{\rm ver}}_{(s,t)}(\xi)=
		\begin{cases}
			tC_\mathcal G^{(\frac{u-p}{u(1-p)})}(\xi)-sC_\mathcal B^{(u)}(-\xi),\qquad &\text{if }\xi\in[0,\log\frac{u(1-p)}{p(1-u)}) \text{ and }t> k^{(u)}(-\xi)s,\\
			\Lambda_{(s,t)}(\xi),\qquad &\text{if }\xi\in[0,\log\frac{u(1-p)}{p(1-u)}) \text{ and }t\leq k^{(u)}(-\xi)s,\\
				\infty,\qquad &\text{if }\xi\in[\log\frac{u(1-p)}{p(1-u)},\infty).
		\end{cases}
		\ee
	\end{enumerate}
\end{theorem}

The last theorem proves the full l.m.g.f. for the boundary model. Define 
\be\label{mm}
\ell^{(u)}(\xi)=\frac{C^{(u)}_\mathcal B(\xi)+C^{(u)}_\mathcal B(-\xi)}{C^{(\frac{u-p}{u(1-p)})}_\mathcal G(\xi)+C^{(\frac{u-p}{u(1-p)})}_\mathcal G(-\xi)}.
\ee
%and 
%\be
%q^{(u)}(\xi)=\Big((k^{(u)}(\xi))\wedge \ell^{(u)}(\xi)\Big)\vee (k^{(u)}(-\xi))
%\ee
%
Then, the l.g.m.f.\ for the boundary last passage time is given by 

\begin{theorem} \label{thm:boundarywhole}
Let $s,t\geq 0$ and $u\in(p,1]$. Then the limit in \eqref{cmgfb} exists for $\xi\geq0$ and is given by
\be\label{eq:lambound}
\Lambda^{(u)}_{(s,t)}(\xi)=
\begin{cases}
	sC^{(u)}_\mathcal B(\xi)-tC^{(\frac{u-p}{u(1-p)})}_\mathcal G(-\xi)	,\qquad &\text{if }\xi\in[0,\log\frac{u(1-p)}{p(1-u)}) \text{ and }t< \ell^{(u)}(\xi)s,\\
		tC_\mathcal G^{(\frac{u-p}{u(1-p)})}(\xi)-sC_\mathcal B^{(u)}(-\xi),\qquad &\text{if }\xi\in[0,\log\frac{u(1-p)}{p(1-u)}) \text{ and }t\geq \ell^{(u)}(\xi)s,\\
				\infty,\qquad &\text{if }\xi\in[\log\frac{u(1-p)}{p(1-u)},\infty).
\end{cases}
\ee

\end{theorem}

Before the two proofs, we begin by verifying the existence of limits \eqref{lamhor} and \eqref{lamver}. We begin by noting that similar arguments as in Lemma \ref{Rlem} give 
that 
\be\label{Jhor}
-\lim_{N\to\infty}N^{-1}\log\P\{G^{(u),\text{hor}}_{\fl{Ns},\fl{Nt}}\geq Nr\}=\inf_{a\in[0,s]} \inf_{x \in \R}\{aI_{\mathcal B}^{(u)}((r-x)/a) + J_{s-a,t}(x)\}.
\ee
Equation \eqref{Jhor} in particular verifies the existence of the limit in the left-hand side, and we denote it by 
\be \label{eq:horj}
-\lim_{N\to\infty}N^{-1}\log\P\{G^{(u),\text{hor}}_{\fl{Ns},\fl{Nt}}\geq Nr\} = J^{(u),\text{hor}}_{s,t}(r).
\ee
Finally, observe that we take the Legendre transform, equation \eqref{Jhor} becomes 
\be \label{eq:jhor*}
(J^{(u),\text{hor}}_{s,t})^*(\xi) = \sup_{a\in[0,s]}\{aC^{(u)}_\mathcal B(\xi)+J^*_{s-a,t}(\xi)\}.
\ee
Symmetric definitions and arguments give similar equations for $J^{(u),\text{ver}}_{s,t}$.
\begin{lemma}\label{justif}
Let $G^{(u),\text{\rm hor}}_{\fl{Ns},\fl{Nt}}$ be the last passage time given by \eqref{lpthor}, and let  $(J^{(u),\text{\rm hor}}_{s,t})^*(\xi) $  given by \eqref{eq:jhor*}. 
Then for $\xi > 0$,
\be \label{eq:horl}
\lim_{N\to\infty}N^{-1}\log\E[e^{\xi G^{(u),\text{\rm hor}}_{\fl{Ns},\fl{Nt}}}]=(J^{(u),\text{\rm hor}}_{s,t})^*(\xi) .
\ee
Corresponding statements hold for $G^{(u),\text{\rm ver}}_{\fl{Ns},\fl{Nt}}$.
\end{lemma}

\begin{proof}
Let $\xi\geq0$. Set
\[
\barbelow{\gamma}=\varliminf_{N\to\infty}N^{-1}\log\E[e^{\xi G^{(u),\text{hor}}_{\fl{Ns},\fl{Nt}}}]\qquad\text{and}\qquad \bar{\gamma}=\varlimsup_{N\to\infty}N^{-1}\log\E[e^{\xi G^{(u),\text{hor}}_{\fl{Ns},\fl{Nt}}}].
\]
The lower bound is immediate using the exponential Chebyshev inequality
\[
N^{-1}\log\P\{ G^{(u),\text{hor}}_{\fl{Ns},\fl{Nt}}\geq Nr\}\leq -\xi r +N^{-1}\log\E[e^{\xi G^{(u),\text{hor}}_{\fl{Ns},\fl{Nt}}}].
\]
Letting $N\to\infty$ along a suitable subsequence gives $\barbelow{\gamma}\geq\xi r -J^{(u),\text{hor}}_{s,t}(r)$ for all $r\in[0,s]$. Thus $\barbelow{\gamma}\geq  (J^{(u),\text{hor}}_{s,t})^*(\xi) $ holds.

For the upper bound we first claim that for every $r>s$
\be\label{claimlem}
\varlimsup_{N\to\infty}N^{-1}\log\E[e^{\xi G^{(u),\text{hor}}_{\fl{Ns},\fl{Nt}}}\mathbbm{1}\{G^{(u),\text{hor}}_{\fl{Ns},\fl{Nt}}\geq Nr\}]=-\infty.
\ee
To see this, apply Holder's inequality to the expectation in \eqref{claimlem}. For any $\alpha>1$,
\begin{align*}
N^{-1}&\log\E[e^{\xi G^{(u),\text{hor}}_{\fl{Ns},\fl{Nt}}}\mathbbm{1}\{G^{(u),\text{hor}}_{\fl{Ns},\fl{Nt}}\geq Nr\}]\\
&\leq N^{-1}\log\{\E[e^{\alpha\xi G^{(u),\text{hor}}_{\fl{Ns},\fl{Nt}}}]^{\alpha^{-1}}\E[\mathbbm{1}\{G^{(u),\text{hor}}_{\fl{Ns},\fl{Nt}}\geq Nr\}^{\frac{\alpha}{\alpha-1}}]^{\frac{\alpha-1}{\alpha}}\}\\
&=( \alpha N)^{-1}\log(\E[e^{\alpha\xi G^{(u),\text{hor}}_{\fl{Ns},\fl{Nt}}}])+(\alpha-1)\alpha^{-1}N^{-1}\log\P\{G^{(u),\text{hor}}_{\fl{Ns},\fl{Nt}}\geq Nr\}.
\end{align*}
The first term is finite since $G^{(u),\text{hor}}_{\fl{Ns},\fl{Nt}}\leq \fl{Ns}$ and for the same reason the second term equals $- \infty$.
%use the logarithm properties, the power property of the indicator function $\mathbbm{1}\{A\}^\beta=\mathbbm{1}\{A\}$ for every $A$ and $\beta$ and finally $\E[\mathbbm{1}\{A\}]=\P\{A\}$
%\begin{align*}
%&=( \alpha N)^{-1}\log(\E[e^{\alpha\xi G^{(u),\text{hor}}_{\fl{Ns},\fl{Nt}}}])+(\alpha-1)\alpha^{-1}N^{-1}\log\P\{G^{(u),\text{hor}}_{\fl{Ns},\fl{Nt}}\geq Nr\}.
%\end{align*}
%We treat the two terms separately. For an upper bound of the first term, we use the fact that  $G^{(u),\text{hor}}_{\fl{Ns},\fl{Nt}}\leq Ns$ and  assume $\alpha$ is sufficiently close to 1 so that  $C_\mathcal G^{(\frac{u-p}{u(1-p)})}(\alpha\xi) < \infty$. Then, 
%\[
%\lim_{N\to\infty}(\alpha N)^{-1}\log(\E[e^{\alpha\xi G^{(u),\text{hor}}_{\fl{Ns},\fl{Nt}}}])\leq \alpha sC_\mathcal B^{(u)}(\alpha\xi)+\alpha tC_\mathcal G^{(\frac{u-p}{u(1-p)})}(\alpha\xi)=C_1<\infty\qquad \forall \xi> 0.
%\]
%For the second addend since the first step is horizontal and the maximal path can collect at most 1 for every right step and once it enters into the bulk it doesn't collect anything at each up step we have that
%\[
%\P\{G^{(u),\text{hor}}_{\fl{Ns},\fl{Nt}}\geq Nr\}
%\begin{cases}
%\in[0,1]\qquad&\text{if }r\in[0,s],\\
%=0&\text{otherwise}.
%\end{cases}
%\]
%Therefore we conclude
%\be
%\varlimsup_{N\to\infty}N^{-1}\log\E[e^{\xi G^{(u),\text{hor}}_{\fl{Ns},\fl{Nt}}}\mathbbm{1}\{G^{(u),\text{hor}}_{\fl{Ns},\fl{Nt}}\geq Nr\}]\leq C_1-C_2J^{(u),\text{hor}}(r)
%\ee
%for positive constants $C_1,C_2$. For every $r>s$, $J^{(u),\text{hor}}(r)=\infty$.
%%
%
%
%Assume that \eqref{claimlem} holds.

 To show the upper bound in \eqref{eq:horl} pick a $\delta>0$ and partition $\R$ with $r_i=i\delta$, $i\in\Z$:
\begin{align}
&N^{-1}\log\E[e^{\xi G^{(u),\text{hor}}_{\fl{Ns},\fl{Nt}}}]\notag\\
&\begin{aligned}
\leq N^{-1}\log\Big[&\sum_{i=-m}^me^{N\xi r_{i+1}}\P\{G^{(u),\text{hor}}_{\fl{Ns},\fl{Nt}}\geq Nr_i\}\\
&\phantom{xxxxxxx} +e^{N\xi r_{-m}}+\E[e^{\xi G^{(u),\text{hor}}_{\fl{Ns},\fl{Nt}}}\mathbbm{1}\{G^{(u),\text{hor}}_{\fl{Ns},\fl{Nt}}\geq Nr_m\}]\Big].
\end{aligned}\label{subseq}
\end{align}
By \eqref{claimlem}, for each $M>0$ there exists $m=m(M)$ so that for all $N$ large enough 
\[ N^{-1}\log \E[e^{\xi G^{(u),\text{hor}}_{\fl{Ns},\fl{Nt}}}\mathbbm{1}\{G^{(u),\text{hor}}_{\fl{Ns},\fl{Nt}}\geq Nr_m\}]<-M.\]
Take a limit as $N \to \infty$ along any subsequence that achieves $\bar \gamma$ to see that  \eqref{subseq} implies
\begin{align*}
 \bar{\gamma}&\leq\max_{-m\leq i\leq m}\{\xi r_{i+1}-J^{(u),\text{hor}}_{s,t}(r_i)\}\vee \xi r_{-m}\vee (-M)\\
&\leq \Big(\sup_{r\in[0,s]}\{\xi r-J^{(u),\text{hor}}_{s,t}(r)\}+\xi\delta\Big)\vee \xi r_{-m}\vee (-M).
\end{align*}
The statement of the Lemma follows by letting $\delta\to0$, $m\to\infty$ and $M\to \infty$.

In order to repeat the estimates for $G^{(u),\text{ver}}_{\fl{Ns},\fl{Nt}}$ the equivalent statement for \eqref{claimlem} is  
\[
\lim_{r \to \infty} \varlimsup_{N\to\infty}N^{-1}\log\E[e^{\xi G^{(u),\text{ver}}_{\fl{Ns},\fl{Nt}}}\mathbbm{1}\{G^{(u),\text{ver}}_{\fl{Ns},\fl{Nt}}\geq Nr\}]=-\infty.
\]
We omit the remaining details; the interested reader can find a similar calculation in \cite{Geo-Sep-13-}.
\end{proof}

\begin{proof}[Proof of Theorem \ref{thm:boundary}]  

The existence of limit \eqref{lamhor} is verified by Lemma \ref{justif}. 
% first note that from definition \eqref{limgppu} follows similar bounds to the ones found in \eqref{ineq1} and  \eqref{ineq2}. Then using similar arguments to the ones used in the proof of Lemma \ref{Rlem}, we get the right tail LDP
%\be\label{Jhor}
%\lim_{N\to\infty}N^{-1}\log\P\{G^{(u),\text{hor}}_{0,(\fl{Ns},\fl{Nt})}\geq Nr\}=-J^{(u),\text{hor}}(r)=-\inf_{a\in[0,s]}\{sI_{\mathcal B}^{(u)}(r/s) \Box J_{(s-a,t)}\}(r).
%\ee
%where $I_{\mathcal B}^{(u)}(r)$ is the rate function in \eqref{eq:berrate}. For the moment suppose that $J^{*(u),\text{hor}}(\xi)$ satisfies a full LDP which will be proven in Lemma \ref{justif} below. Then by  Varadhan's theorem for $\xi\geq0$ the logarithm generating function in \eqref{lamhor} is given by
Then, use in sequence equations \eqref{eq:jhor*} and \eqref{eq:horl} and Proposition \ref{propJ} to write
%
% $\Lambda^{(u),\text{hor}}_{(s,t)}(\xi)=J^{*(u),\text{hor}}(\xi)$.  Therefore using the same arguments as in \eqref{Rstara} and substituting the result in Proposition \ref{propJ} we obtain
%
\begin{align*}
\Lambda^{(u),\text{hor}}_{(s,t)}(\xi)&=\sup_{a\in[0,s]}\{aC^{(u)}_\mathcal B(\xi)+J^*_{s-a,t}(\xi)\}\\
&=\sup_{a\in[0,s]}\Big\{\inf_{\theta\in(p,1]}\{a\Big(C^{(u)}_\mathcal B(\xi)-C^{(\theta)}_\mathcal B(\xi)\Big)+sC^{(\theta)}_\mathcal B(\xi)-tC^{(\frac{\theta-p}{\theta(1-p)})}_\mathcal G(-\xi)\}\Big\}.
\end{align*}
The sup and inf can be interchanged by a minimax theorem (e.g. \cite{Kas-94}). The function inside the supremum is linear in $a$. Thus  the supremum will be reached at one of the two boundary points according to the sign of the difference 
\[
C^{(u)}_\mathcal B(\xi)-C^{(\theta)}_\mathcal B(\xi)
\begin{cases}
>0,\qquad&\text{if }\theta \in(u,1],\\
=0, \qquad &\text{if } \theta = u,\\
<0,\qquad&\text{if }\theta \in(p,u).
\end{cases}
\]
Therefore we have 
\be\label{lammin}
\begin{aligned}
\Lambda^{(u),\text{hor}}_{(s,t)}(\xi)=\inf_{\theta\in(u,1]}\{sC^{(u)}_\mathcal B(\xi)&-tC^{(\frac{\theta-p}{\theta(1-p)})}_\mathcal G(-\xi)\}\wedge
\{ sC^{(u)}_\mathcal B(\xi)-tC^{(\frac{u-p}{u(1-p)})}_\mathcal G(-\xi)\}\\
&\hspace{2.2cm}\wedge\inf_{\theta\in(p,u)}\{sC^{(\theta)}_\mathcal B(\xi)-tC^{(\frac{\theta-p}{\theta(1-p)})}_\mathcal G(-\xi)\}.
\end{aligned}
\ee
Note that, since $-C^{(\frac{\theta-p}{\theta(1-p)})}_\mathcal G(-\xi)$ is increasing in $\theta$, the first term on the right-hand side of \eqref{lammin} is always greater than the second one. So, it remains to compare the second and the third term. 

Call $\theta^*$ the minimizing point in $(p,1]$ for the expression $sC^{(\theta)}_\mathcal B(\xi)-tC^{(\frac{\theta-p}{\theta(1-p)})}_\mathcal G(-\xi$ \eqref{eq:umin} in this specific case. Then, there are two possible cases:
\begin{enumerate}[(1)] 
\item If $\theta^*\le u$, then 
\[ \Lambda^{(u),\text{hor}}_{(s,t)}(\xi)=\inf_{\theta\in(p,u)}\{sC^{(\theta)}_\mathcal B(\xi)-tC^{(\frac{\theta-p}{\theta(1-p)})}_\mathcal G(-\xi)\}=sC^{(\theta^*)}_\mathcal B(\xi)-tC^{(\frac{\theta^*-p}{\theta^*(1-p)})}_\mathcal G(-\xi)=\Lambda_{(s,t)}(\xi).\] %, which corresponds to the case when $s\frac{\partial C^{(u)}_\mathcal B(\xi)}{\partial u}-t\frac{\partial C^{(\frac{u-p}{u(1-p)})}_\mathcal G(-\xi)}{\partial u}>0$. 
%\item If $\theta^*= u$ then $\Lambda^{(u),\text{hor}}_{(s,t)}(\xi)=sC^{(\theta^*)}_\mathcal B(\xi)-tC^{(\frac{\theta^*-p}{\theta^*(1-p)})}_\mathcal G(-\xi)=\Lambda_{(s,t)}(\xi) $. % if $\theta^*<u$ which correspond to the case when $s\frac{\partial C^{(u)}_\mathcal B(\xi)}{\partial u}-t\frac{\partial C^{(\frac{u-p}{u(1-p)})}_\mathcal G(-\xi)}{\partial u}=0$. 
\item If $\theta^* > u$ then 
\[  \Lambda^{(u),\text{hor}}_{(s,t)}(\xi)=sC^{(u)}_\mathcal B(\xi)-tC^{(\frac{u-p}{u(1-p)})}_\mathcal G(-\xi).\] 
\end{enumerate}
This concludes the proof of \eqref{lamhorc}. 
For the analogous result in the vertical case, first note that we may write  
\be\label{eq:vercoun}
\Lambda_{(s,t)}(\xi) = J^*_{s,t}(\xi)=\inf_{u\in(p,1]}\{tC_\mathcal G^{(\frac{u-p}{u(1-p)})}(\xi)-sC_\mathcal B^{(u)}(-\xi)\}.
\ee
That is possible to prove either by repeating the same computation in the subsection \ref{subsec:ecom} but starting $ G^{(u)}_{\fl{Ns},\fl{Nt}}-G^{(u)}_{\fl{Ns},0}=\sum_{j=1}^{\fl{Nt}}J^{(u)}_{\fl{Ns},j}$, or by computing \eqref{eq:vercoun} as in the proof of Theorem \ref{thm:lam} and observe that it gives the same result.

Then as in the case for the horizontal boundary only,
\begin{align*}
\Lambda^{(u),\text{ver}}_{(s,t)}(\xi)&=\sup_{a\in[0,t]}\{aC_\mathcal G^{(\frac{u-p}{u(1-p)})}(\xi)+J^*_{s,t-a}(\xi)\}\\
&=\sup_{a\in[0,t]}\Big\{\inf_{\theta\in(p,1]}\{a\Big(C_\mathcal G^{(\frac{u-p}{u(1-p)})}(\xi)-C_\mathcal G^{(\frac{\theta-p}{\theta(1-p)})}(\xi)\Big)+tC_\mathcal G^{(\frac{\theta-p}{\theta(1-p)})}(\xi)-sC_\mathcal B^{(\theta)}(-\xi)\}\Big\}.
\end{align*}
From this expression we see that we need to restrict $\xi\in[0,\log\frac{u(1-p)}{p(1-u)})$, otherwise $\Lambda^{(u),\text{ver}}_{(s,t)}(\xi)$ is not finite.
Then, as before, for $\xi\in[0,\log\frac{u(1-p)}{p(1-u)})$
\begin{align*}
&\Lambda^{(u),\text{ver}}_{(s,t)}(\xi)\\
&\hspace{0.6cm}=
\begin{cases}
\inf_{\theta\in(p,1]}\{tC^{(\frac{\theta-p}{\theta(1-p)})}_\mathcal G(\xi)-sC^{(\theta)}_\mathcal B(-\xi)\}= \Lambda_{(s,t)}(\xi)\qquad&\text{if }t\leq k^{(u)}(-\xi)s,\\
\inf_{\theta\in(p,u]}\{tC^{(\frac{\theta-p}{\theta(1-p)})}_\mathcal G(\xi)-sC^{(u)}_\mathcal B(-\xi)\}=tC^{(\frac{u-p}{u(1-p)})}_\mathcal G(\xi)-sC^{(u)}_\mathcal B(-\xi)\qquad&\text{if }t> k^{(u)}(-\xi)s.
\end{cases}
\end{align*}
This concludes the proof of the theorem.
\end{proof}

\begin{proof}[Proof of Theorem \ref{thm:boundarywhole}]
All the proof is based on \eqref{lambound}. First note that by Proposition \ref{propJ} and \eqref{eq:vercoun} we have that for any $u$
\be\label{dom}
\Lambda_{(s,t)}(\xi) \leq sC_\mathcal B^{(u)}(\xi)-tC_\mathcal G^{(\frac{u-p}{u(1-p)})}(-\xi)\quad\text{and} \quad \Lambda_{(s,t)}(\xi) \leq tC_\mathcal G^{(\frac{u-p}{u(1-p)})}(-\xi)-sC_\mathcal B^{(u)}(-\xi).
\ee
Therefore, if $\xi\in[\log\frac{u(1-p)}{p(1-u)}),\infty)$, $ \Lambda^{(u)}_{(s,t)}(\xi)=\infty$. 

If $\xi\in(0,\log\frac{u(1-p)}{p(1-u)})$ we define three regions in the quadrant by 
\[
L = \{ (s,t): t < k^{(u)}(-\xi) s\},\,\, M =  \{ (s,t): k^{(u)}(-\xi) s \le t \le k^{(u)}(\xi) s\}, \,\, U = \R^2_+\setminus (M\cup L).
\]
%
%
%\be\label{lmgft2}
% t< k^{(u)}(\xi)s, \quad \text{and} \quad 
%\ee
%and 
%\be\label{lmgft3}
%t> k^{(u)}(-\xi)s
%\ee
%define three ranges for $(s,t)$:
%\begin{enumerate}[(1)]
%\item  \eqref{lmgft2} and \eqref{lmgft3} hold both if and only if $(k^{(u)}(\xi))^{-1}t\leq s\leq (k^{(u)}(-\xi))^{-1}t$;
%\item\eqref{lmgft2} holds and \eqref{lmgft3} fails if and only if $ s> (k^{(u)}(-\xi))^{-1}t$;
%\item\eqref{lmgft2} fails and \eqref{lmgft3} holds if and only if $ s< (k^{(u)}(\xi))^{-1}t$.
%\end{enumerate}
$k^{(u)}(\xi)$ is defined by \eqref{eq:ku} and one can directly verify that $k^{(u)}(-\xi)< k^{(u)}(\xi)$. 
For $(s,t) \in L$, $\Lambda^{(u)}_{s,t}(\xi)=sC^{(u)}_\mathcal B(\xi)-tC^{(\frac{u-p}{u(1-p)})}_\mathcal G(-\xi) = \Lambda^{(u),\rm hor}_{(s,t)}(\xi) $ by \eqref{lambound},\eqref{dom}, since $\Lambda^{(u),\rm ver}_{(s,t)}(\xi) = \Lambda_{(s,t)}(\xi)$. For $(s,t)\in U$ the arguments are symmetric, with $\Lambda^{(u)}_{(s,t)}(\xi)=tC^{(\frac{u-p}{u(1-p)})}_\mathcal G(\xi)-sC^{(u)}_\mathcal B(-\xi)$.

From \eqref{lambound}, \eqref{dom} and Theorem \ref{thm:boundary}, we have that 
\[
\Lambda^{(u)}_{(s,t)}(\xi) = 
\begin{cases}
\Lambda^{(u), \rm ver}_{(s,t)}(\xi), &\quad t \ge k^{(u)}(\xi)s,\\
\Lambda^{(u), \rm ver}_{(s,t)}(\xi)\vee \Lambda^{(u), \rm hor}_{(s,t)}(\xi), &\quad k^{(u)}(\xi) s< t < k^{(u)}(\xi)s,\\
\Lambda^{(u), \rm hor}_{(s,t)}(\xi), &\quad t \le k^{(u)}(-\xi)s.
\end{cases}
\]
By \eqref{lambound} and Theorem \ref{thm:boundary},  $\Lambda^{(u)}_{(s,t)}(\xi) $ is continuous in $(s,t)$. From this  and the fact that the middle branch above is linear in 
$(s,t)$, we conclude that the slope $ \ell^{(u)}(\xi) $ of the line 
\[
t = \ell^{(u)}(\xi) s  \Longleftrightarrow \{ (s,t) \in \R^2_+:  \Lambda^{(u), \rm ver}_{(s,t)}(\xi) = \Lambda^{(u), \rm hor}_{(s,t)}(\xi)\} 
\]
satisfies $k^{(u)}(\xi) \ge \ell^{(u)}(\xi) \ge k^{(u)}(-\xi)$ and therefore 
\[
\Lambda^{(u)}_{(s,t)}(\xi)=
\begin{cases}
sC^{(u)}_\mathcal B(\xi)-tC^{(\frac{u-p}{u(1-p)})}_\mathcal G(-\xi), &\quad \text{ if }  k^{(u)}(-\xi)) s \le t \le \ell^{(u)}(\xi) s,\\
tC^{(\frac{u-p}{u(1-p)})}_\mathcal G(\xi)-sC^{(u)}_\mathcal B(-\xi), &\quad \text{ if }  \ell^{(u)}(\xi) s < t \le  k^{(u)}(\xi)) s. 
\end{cases}
\]
This gives the theorem.
\end{proof}

\appendix
\section{A convex analysis proposition}
\label{app:conv}

\begin{proof}[Proof of Proposition \ref{prop:dousup}]
Fix $(s,t)\in \R_+^2$ and call $\nu=\frac{t}{s}$.  Observe that under the hypotheses of this proposition there exists a unique $u^*_{s,t}=\argmin_{u\in I}f_{s,t}(u)=u^*_{1,\nu}$. This minimum point can be eventually reached at $u^*_{s,t}=b$ if $f_{s,t}'(u)\leq0$ for all $u \in I$. In particular,  $u^*_{s,t}$ solves the equation  
\be \label{eq:devc}
f_{s,t}'(u)=sh'(u)+tg'(u)=0\implies t=-\frac{h'(u)}{g'(u)}s.
\ee
The largest value $-\frac{h'(u)}{g'(u)}$ can take is when $u = b$. For any $(s, t)$ above the line  
\be\label{cl}
t=-\frac{h'(b)}{g'(b)}s,
\ee
equation \eqref{eq:devc} has no solution and in fact $f_{s,t}'(u) < 0$ and $\argmin f_{s,t}(u) = b$. For any $(s, t)$ below this line \eqref{cl} a solution to \eqref{eq:devc} exists and is giving the minimizing argument. We call the line \eqref{cl} \textit{ the critical line}. 

The identity in \eqref{ii} implies that for all $z\in[0,s]$ and $\tilde{z}\in[0,t]$ the following inequalities hold 
\[
\Lambda(s-z,t)\leq f_{s-z,t}(u^*_{s-z,t}),\quad \Lambda(s,t-\tilde{z})\leq f_{s,t-\tilde{z}}(u^*_{s,t-\tilde{z}}).
\]
Fix a  $u\in I$ and subtract, from each side of the inequalities above, $f_{s-z,t}(u)$ and $f_{s,t-\tilde{z}}(u)$ respectively, to obtain 
\begin{align}
\Lambda(s-z,t)-f_{s-z,t}(u)&\leq f_{s-z,t}(u^*_{s-z,t})-f_{s-z,t}(u),\label{eq:sz}\\
\Lambda(s,t-\tilde{z})- f_{s,t-\tilde{z}}(u)&\leq f_{s,t-\tilde{z}}(u^*_{s,t-\tilde{z}})- f_{s,t-\tilde{z}}(u).\label{eq:sz'}
\end{align}
Since the minimizer is unique we have that $f_{s-z,t}(u^*_{s-z,t})-f_{s-z,t}(u)<0$ unless $u=u^*_{s-z,t}$ and $f_{s,t-\tilde{z}}(u^*_{s,t-\tilde{z}})- f_{s,t-\tilde{z}}(u)<0$ unless $u=u^*_{s,t-\tilde{z}}$. 
Set $u=u^*_{s,t}$ and substitute it in \eqref{eq:sz} and  \eqref{eq:sz'} 
\begin{align}
\Lambda(s-z,t)-f_{s-z,t}(u^*_{s,t})&\leq f_{s-z,t}(u^*_{s-z,t})-f_{s-z,t}(u^*_{s,t}),\label{eq:sz2}\\
\Lambda(s,t-\tilde{z})- f_{s,t-\tilde{z}}(u^*_{s,t})&\leq f_{s,t-\tilde{z}}(u^*_{s,t-\tilde{z}})- f_{s,t-\tilde{z}}(u^*_{s,t}).\label{eq:sz'2}
\end{align}
Note that \eqref{ii} implies that there exists a sequence $z_n\to z\in[0,s]$ or $\tilde{z}_n\to \tilde{z}\in[0,t]$ such that at least one of the following limits holds
\begin{align}
\Lambda(s-z_n,t)-f_{s-z_n,t}(u^*_{s,t})\to0,\label{zz}\\
 \Lambda(s,t-\tilde z_n)- f_{s,t-z_n'}(u^*_{s,t})\to0.\label{z}
\end{align}

If $t<-\frac{h'(b)}{g'(b)}s$ then the point $(s,t-\tilde{z})$ is below the critical line  for every $\tilde z\in [0,t]$. The point $(s-z,t)$ can be above or below the critical line according to the value of $z$. 
We analyse these two cases for the first supremum in \eqref{ii}. The case for the second supremum is identical to case (a) below, as for all $\tilde z$, the index point stays below the critical line.

\begin{enumerate}[(a)]
\item If $0\leq z <s+t\frac{g'(b)}{h'(b)}$, we have that both $u^*_{s,t},u^*_{s-z,t}\in (a,b)$. In particular 
\be \label{hg}
h'(u^*_{1,\nu})+\nu g'(u^*_{1,\nu})=0.
\ee
By the implicit function theorem we can take the derivative of the previous expression respect to $\nu$ and find 
\[
\frac{du^*_{1,\nu}}{d\nu}=-\frac{g'(u^*_{1,\nu})}{h''(u^*_{1,\nu})+\nu g''(u^*_{1,\nu})}>0.
\]
This implies that for all $z\in (0,s+t\frac{g'(b)}{h'(b)})$ and $\tilde{z}\in (0,t)$, $u^*_{s,t-\tilde{z}}<u^*_{s,t}<u^*_{s-z,t}$. 

We want to show that \eqref{zz} is possible if only $z_n\to0$  from which the result follows from continuity. The right hand side in \eqref{eq:sz2} is negative and therefore, by continuity we can argue that the supremum will be attained at one of the boundary points.
 Thus, we have only to show that
\begin{align}
\lim_{z\nearrow s+t\frac{g'(b)}{h'(b)}} f_{s-z,t}(u^*_{s-z,t})-f_{s-z,t}(u^*_{s,t})&<0.\label{lim1}
\end{align}
 For any fixed $z\in (0,s+t\frac{g'(b)}{h'(b)})$ we have that 
\[
f_{s-z,t}(u^*_{s-z,t})-f_{s-z,t}(u^*_{s,t})<0.
\]
Therefore we obtain the proof if we show that the last expression is decreasing in $z$. 
Take the derivative in $z$, use \eqref{hg}, recall that $u^*_{s,t}<u^*_{s-z,t}$ and $h(u)$ is an increasing function by hypothesis
\begin{align*}
&\frac{d}{dz}\Big((s-z)h(u^*_{s-z,t})+tg(u^*_{s-z,t})-[(s-z)h(u^*_{s,t})+tg(u^*_{s,t})]\Big)\\
&=-h(u^*_{s-z,t})+\Big((s-z)h'(u^*_{s-z,t})+tg'(u^*_{s-z,t})\Big)\frac{du^*_{s-z,t}}{dz}+h(u^*_{s,t})\\
&=h(u^*_{s,t})-h(u^*_{s-z,t})<0.
\end{align*}

\item If $s+t\frac{g'(b)}{h'(b)}\leq z\leq s$, we have that $u^*_{s-z,t}=b$. Note that $u^*_{s,t}<u^*_{s-z,t}=b$ in this case and therefore 
$f_{s-z,t}(b)-f_{s-z,t}(u^*_{s,t})<0$ for every $z\in[s+t\frac{g'(b)}{h'(b)},s]$. This implies that \eqref{zz} can never be true for $z \in (s+t\frac{g'(b)}{h'(b)}, s]$. But the boundary point $z = s+t\frac{g'(b)}{h'(b)}$ is also not optimal by continuity considerations and \eqref{lim1}. 
\end{enumerate}

Therefore, the potential maximum happens at $z = 0$. Similarly, this will be true for $\tilde z = 0$ and therefore $\Lambda(s,t) = f_{s, t}(u^*_{s,t})$ as required.
\end{proof}

\section{Law of large numbers and proof of Burke's property}
\label{app: LLN} 

\begin{proof}[Proof of Lemma \ref{burke}] 
We omit the superscripts and indices from the $I, J$ and we simply denote 
\[
\tilde I = \max\{I-J, \om\}, \quad \text{and} \quad \tilde J = (J-I+\om)^+.
\]
The marginal distributions of $(\tilde I, \tilde J, \alpha)$ can be computed directly, using equations \eqref{eq:4}, \eqref{eq:5}. For example, since $\alpha$ only takes the values $0$ or $1$ it suffices to compute 
\begin{align*}
\mathbb{P}\{\alpha=1\}&=\mathbb{P}\{ \min\{I,J+\om\}=1\}=\mathbb{P}\{ I=1, J+\om\geq1\} \\
%&=\mathbb{P}\{ I=1\}\Big(\mathbb{P}\{\om=1|J\geq0\}\mathbb{P}\{ J\geq0\}+\mathbb{P}\{ \om=0|J\geq1\}\mathbb{P}\{ J\geq1\}\Big)\\
&=u\Big(p+(1-p)\Big(1- \frac{u-p}{u(1-p)} \Big)\Big)=p.
\end{align*}
The remaining calculations are left to the reader. 
 
The proof of independence goes by calculating the Laplace transform 
of the triple $(\tilde I, \tilde J, \alpha)$. Let $x \in \R, z \in \R$ and  
$y > \log[p(1-u)/(u(1-p))]$.  Recall that $u \in (p,1]$. Then compute, using \eqref{eq:4}and \eqref{eq:5}, the joint Laplace transform 
\allowdisplaybreaks
\begin{align*}
\E&(e^{-x \tilde I - y\tilde J - z\alpha}) = \mathbb{E}[e^{-x\max\{I-J,\om\}-y(J-I+\om)^+-z\min\{I,J+\om\}}]\\
%&=
%p\mathbb{E}[e^{-x\max\{I-J,1\}-y(J-I+1)^+-z\min\{I,J+1\}}]+q\mathbb{E}[e^{-x(I-J)^+-y(J-I)^+-z(I\wedge J)}]\\
&=pu\mathbb{E}[e^{-x(\max\{1-J,1\})-yJ-z\min\{1,J+1\}}]+p(1-u)\mathbb{E}[e^{-x-y(J+1)^+}]\\
&\phantom{xxxxxxxxxxxxx}+(1-p)u\mathbb{E}[e^{-x(1-J)^+-y(J-1)^+-z(1\wedge J)}]+(1-p)(1-u)\mathbb{E}[e^{-yJ}]\\
&=pu\frac{u-p}{u(1-p)}e^{-(x+z)}\sum_{j=0}^\infty\bigg(\frac{p(1-u)}{u(1-p)}\bigg)^je^{-yj}\\
&\phantom{xxxxx}+p(1-u)\frac{u-p}{u(1-p)}e^{-(x+y)}\sum_{j=0}^\infty\bigg(\frac{p(1-u)}{u(1-p)}\bigg)^je^{-yj}\\
&\phantom{xxxxxxxxxx}+(1-p)u\frac{u-p}{u(1-p)}\bigg(e^{-x}+\sum_{j=1}^\infty\bigg(\frac{p(1-u)}{u(1-p)}\bigg)^je^{-y(j-1)-z}\bigg)\\
&\phantom{xxxxxxxxxxxxxxxxxxxxxxxx}+(1-p)(1-u)\sum_{j=0}^\infty\bigg(\frac{p(1-u)}{u(1-p)}\bigg)^je^{-yj}\\
&=\frac{\frac{u-p}{u(1-p)}}{1-\frac{p(1-u)}{u(1-p)}e^{-y}}\bigg(pue^{-(x+z)}+p(1-u)e^{-(x+y)}+(1-p)(1-u)\bigg)\\
&\phantom{xxxxx}+(1-p)u\frac{\frac{u-p}{u(1-p)}}{1-\frac{p(1-u)}{u(1-p)}e^{-y}}\bigg[e^{-x}\bigg(1-\frac{p(1-u)}{u(1-p)}e^{-y}\bigg)+e^{-z}\frac{p(1-u)}{u(1-p)}\bigg]\\
&=\frac{\frac{u-p}{u(1-p)}}{1-\frac{p(1-u)}{u(1-p)}e^{-y}}\bigg(pue^{-(x+z)}+ (1-p)(1-u)+ (1-p)ue^{-x}+p(1-u)e^{-z}\bigg)\\
&=\E(e^{-y\tilde J}) \E (e^{-x\tilde I})\E(e^{-z\alpha})\qedhere
\end{align*}
\end{proof}

\begin{proof}[Proof of Corollary  \ref{cor:downrightpath}]
The proof is inductive. Consider the countable set of paths $\Psi$ that connect the $y$-axis to the $x$-axis. The trivial case is when $\mathcal I _{\psi_0}=\emptyset$ (i.e. $\psi_0$ is the union of the two axes,  $\psi_0\in\Psi$) and  then the statement reduces to the independence of the $\om_{i,j}$'s on the $x$ and $y$ axes which is true by the definition of the environment.

Assume that for a $\psi\in \Psi$ the statement holds. We say that a lattice vertex $v_{i_0}$ on $\psi $ $(i,j) \in \Z^2_+$ is a west-south corner  of $\psi$ if 
\[
(v_{ i_0-1},v_{i_0},v_{i_0+1})=((i,j+1),(i,j),(i+1,j)).
\]
Now define a new path $\tilde \psi $ by replacing $v_{i_0}$ with $\tilde v_{i_0}=(i+1,j+1)$ and keep all the other points intact which means that $v_{i}=\tilde v_{i}$ for $i\not =i_0$. In this way we have $\mathcal I _{\tilde\psi}=\mathcal I _{\psi}\cup \{(i,j)\}$.

Going from $\psi$ to $\tilde \psi $ we have also a change in the set of random variables in \eqref{eq:dwrp}. In fact
\be\label{eq:oldvar}
\{ I_{i+1,j},J_{i,j+1}\}
\ee
have been replaced by
\be\label{eq:newvar}
\{ I_{i+1,j+1},J_{i+1,j+1},\alpha_{i+1,j+1}\}.
\ee
By \eqref{eq:4} and \eqref{eq:5} the variables in \eqref{eq:newvar} are determined by \eqref{eq:oldvar} and $\om_{i+1,j+1}$. By construction $\om_{i+1,j+1}$ is independent of \eqref{eq:dwrp} for the $\psi$ under consideration. By construction the triple $\{ I_{i+1,j},J_{i,j+1},\om_{i+1,j+1}\}$ are independent random variables and by the induction assumption we have they are in turn independent of the all other variables \eqref{eq:dwrp}. Finally   Lemma \ref{burke} implies that also the triple $\{ I_{i+1,j+1},J_{i+1,j+1},\om_{i,j}\}$ are independent random variables with the correct marginal distribution and they are independent of all the random variables of $\tilde \psi $. All these observations prove that also $\tilde \psi $ satisfies the statement of the corollary.

Note that if we start with $ \psi_0$, we can build a path $\psi\in \Psi$ by flipping west-south corners finitely many times. The induction argument guarantees that class $\Psi$  satisfies the corollary. 

The general statement follows also for an arbitrary down-right path $\psi$ using the independence of finite subcollections. Consider any square $\mathcal R=\{i\leq 0, j\leq M\}$ large enough so that the corner $(M,M)$ lies outside $\psi\cup\mathcal I _{\psi}$. The $\alpha$ and $L(\psi)$ variables associated to $\psi$ that lie in $\mathcal R$ are a subset of the variables of the path $\tilde\psi$ that goes through the points $(0,M),(M,M)$ and $(M,0)$. This path $\tilde \psi$ connects the axes so the first part of the proof applies to it. Thus the variables \eqref{eq:dwrp} that lie inside an arbitrarily large square are independent.
\end{proof}

\subsection{Laws of large numbers}

\begin{proof}[Proof of Theorems \ref{thm:llnp}, \ref{thm:LLNp}]

Recall the definitions of  $g_{pp}^{(u),\text{ver}}(s,t)$ and  $g_{pp}^{(u),\text{hor}}(s,t)$ from Theorem \ref{thm:HOR}.
In the sequence we use freely the facts that $g_{pp}(s,t)$ is 1-homogeneous and concave.

If $t<\frac{1-p}{p}s$, start from equation \eqref{eq:varform}. Divide by $N$ to obtain the macroscopic variational formulation
\begin{align}
g_{pp}^{(u)}(s,t)&=g_{pp}^{(u),\text{hor}}(s,t)\bigvee g_{pp}^{(u),\text{ver}}(s,t) \notag\\
&=\sup_{0\leq z\leq s}\{g_{pp}^{(u)}(z,0)+g_{pp}(s-z,t)\}\bigvee\sup_{0\leq \tilde{z}\leq t}\{g_{pp}^{(u)}(0, \tilde{z})+g_{pp}(s,t-\tilde{z})\} \notag\\
&= \sup_{0\leq z\leq s}\{z \E(I^{(u)})+g_{pp}(s-z,t)\}\bigvee\sup_{0\leq \tilde{z}\leq t}\{ \tilde{z} \E(J^{(u)})+g_{pp}(s,t-\tilde{z})\}.\label{eq:maxpath}
\end{align}

How we obtain equation \eqref{eq:maxpath} is not immediate to see. Since this step is technical and it is not the main goal of this proof, we postpone it until the end. Therefore assume for the moment that \eqref{eq:maxpath}  holds. Subtract $g_{pp}^{(u)}(s,t)$ to either side of \eqref{eq:maxpath} 
\[
0=\sup_{0\leq z\leq s}\Big\{g_{pp}(s-z,t)-\Big[(s-z)u+t\frac{p(1-u)}{u-p}\Big]\Big\}\bigvee\sup_{0\leq \tilde{z}\leq t}\Big\{ g_{pp}(s,t-\tilde{z})-\Big[(t-\tilde{z})\frac{p(1-u)}{u-p}+su\Big]\Big\}.
\]

We use Proposition \ref{prop:dousup} by identifying as $I=(p,1]$, $\Lambda(s,t)=g_{pp}(s,t)$, $h(u)=s$, $g(u)=\frac{p(1-u)}{u-p}$ and therefore $f_{s,t}(u)=su+t\frac{p(1-u)}{u-p}$. Note that $h'(u)>0$, $g'(u)<0$ for every $u\in(p,1]$ and in particular $f''_{s,t}(u)>0$ for every $(s,t)\in \R^2_+$.  Moreover $\lim_{u\searrow p}f_{s,t}(u)=\infty$ and $f_{s,t}(1)=s<\infty$. Therefore 
\be\label{eq:14}
g_{pp}(s,t)=\min_{u\in(p,1]}\Big\{su+t\frac{p(1-u)}{u-p}\Big\}=\big(\sqrt{ps}+\sqrt{(1-p)t}\big)^2-t,\qquad \text{if }t<s\frac{1-p}{p}.
\ee

If $t\geq \frac{1-p}{p}s$, We want to find an upper and a lower bound for $ G_{\fl{Ns}, \fl{Nt}}$. The upper bound is trivial since by model definition $ G_{\fl{Ns}, \fl{Nt}}\leq \fl{Ns}$. For the lower bound, force a macroscopic distance from the critical line, i.e. assume that it is possible to find a $\e > 0$ so that the sequence of endpoints $(\fl{Ns}, \fl{Nt})$ satisfy
\be \label{eq:flatedge}
\varliminf_{N \to \infty} \frac{\fl{Nt}}{\fl{Ns}}  \ge \frac{1-p}{p} + \e. 
\ee
%\nicos{right remark that when u = 0 or 1, char direction is actually px or x /p}. 
%first consider the situation where the boundary variables  correspond to $u = 0$. In this case, a potential 
% maximal path starts from $(0,0)$, goes through the vertical boundary with probability 1 and it always collects 1 at each step up to $n(N)$ and then it enters into the bulk to go horizontally up to $N$ with 0 gain.  
% 
%On the other hand, for any value of $0 < \lambda <1$ the maximum weight a path can collect, in any realisation is $n(N)$. So assume that we have boundaries with variables corresponding to a parameter $\lambda \in (0,1)$ and 

Then consider the following strategy:  
construct an approximate maximal path $\pi$ for  $G_{\fl{Ns}, \fl{Nt}}$, knowing that for large $ \fl{Nt}\geq (\frac{1-p}{p} + \e)\fl{Ns}$. $\pi$ starts from $(0,0)$ and moves up until it finds a weight  to collect horizontally on his right. After that this procedure repeats. For each iteration of this procedure, the vertical length of this path increases by a random Geometric($p$) length, independently of the past. Define $Y\sim$ Geometric($p$) with range on $0, 1, ...$. 
By construction, we have 
\[
\Big\{ \sum_{i=1}^{\fl{Ns}}Y_i > \fl{Nt} \Big\} \supseteq \{ G_{\fl{Ns}, \fl{Nt}} < \fl{Ns}\}.
\]
The relation on $(s,t)$ implies that the larger event above is large deviation event, and therefore by the Borel-Cantelli lemma,  $G_{\fl{Ns}, \fl{Nt}} = \fl{Ns}$.
Scaling by $N$ and letting it tend to $\infty$ completes the proof.

We finally prove  \eqref{eq:maxpath}.
For a lower bound, fix any $z \in [0, s]$ and $\tilde z \in [0, t]$  . Then if we move on the horizontal axis
\[
G^{(u)}_{\fl{Ns},\fl{Nt}} \ge \sum_{i=1}^{\fl{Nz}} I^{(u)}_{i, 0} + G_{(\fl{Nz}, 1), (\fl{Ns}, \fl{Nt})}. 
\]
Divide by $N$. Observe that the left hand side converges  a.s.\ to $g^{(u)}_{pp}(s,t)$. While  the first term on the right converges a.s.\ to $z \E(I^{u})$. The second on the right, converges in probability to $g_{pp}(s-z, t)$. In particular, we can find a subsequence $N_k$ such that the convergence is almost sure for the second term. Taking limits on this subsequence, we conclude 
\[
g_{pp}^{(u)}(s,t) \ge z \E(I^{u}) + g_{pp}(s-z, t).
\] 
Since $z$ is arbitrary we can take supremum over  $z$ in both sides of the inequality above. 
The same arguments will work if we move on the vertical axis. Thus, we obtain the lower bound for \eqref{eq:maxpath}. 

For the upper bound, we partition the two axes. Fix $\e,\tilde \e>0$ and let $\{ 0 =q_0, \e=q_1, 2\e=q_2, \ldots, s\fl{\e^{-1}}\e,  s=q_M \}$ a partition of $(0,s)$ and  $\{ 0 =q_0, \tilde\e=q_1, 2 \tilde\e=q_2, \ldots, t\fl{ \tilde\e^{-1}} \tilde\e,  t=q_{\tilde M} \}$ a partition of $(0,t)$. 
The maximal path that utilises $G^{(u)}_{N,N}$ has to exit between $\fl{Nk\e}$ and $\fl{N(k+1)\e}$ for some $k$ if it chooses to go through the $x$-axis and between $\fl{N\tilde k\tilde\e}$ and $\fl{N(\tilde k+1)\tilde\e}$ for some $\tilde k$ if it goes through the $y$-axis. Therefore,  
we may write 
\begin{align*}
G^{(u)}_{\fl{Ns},\fl{Nt}}  &\le \max_{0 \le k \le \fl{\e^{-1}}}\bigg\{\sum_{i=1}^{\fl{N(k+1)\e}} I^{(u)}_{i, 0} + G_{(\fl{Nk\e}, 1), (\fl{Ns},\fl{Nt})}\bigg\}\\
&\phantom{xxxxxxxxxxx}\bigvee \max_{0 \le \tilde k \le \fl{\tilde\e^{-1}}}\bigg\{\sum_{j=1}^{\fl{N(\tilde k+1)\tilde\e}} J^{(u)}_{0, j} + G_{(1, (\fl{N\tilde k\tilde\e})), (\fl{Ns},\fl{Nt})}\bigg\}.
\end{align*}
Divide by $N$. The right-hand side converges in probability to the constant 
\begin{align*}
 &\max_{0 \le k \le \fl{\e^{-1}}}\{(k+1)\e u + g_{pp}(s-\e k, t)\}\\
&\phantom{xxxxxxx}\bigvee \max_{0 \le \tilde k \le \fl{\tilde\e^{-1}}}\Big\{(\tilde k+1)\tilde\e \frac{p(1-u)}{u-p} + g_{pp}(s, t-\tilde\e\tilde k) \Big\}\\
&=\Big( \max_{0 \le k \le \fl{\e^{-1}}}\{k\e u + g_{pp}(s-\e k, t)\} + \e u\Big)\\
&\phantom{xxxxxxx}\bigvee \Big(\max_{0 \le \tilde k \le \fl{\tilde\e^{-1}}}\bigg\{\tilde k\tilde\e \frac{p(1-u)}{u-p} + g_{pp}(s, t-\tilde\e\tilde k) \Big\} +\tilde\e\frac{p(1-u)}{u-p}\Big)\\
&=\Big(\max_{q_k }\{q_k u + g_{pp}(s-q_k, t)\} + \e u\Big)\\
&\phantom{xxxxxxx}\bigvee \Big(\max_{q_{\tilde k}}\Big\{q_{\tilde k} \frac{p(1-u)}{u-p} + g_{pp}(s, t-q_{\tilde k}) \Big\} +\tilde\e\frac{p(1-u)}{u-p}\Big)\\
&\le \Big(\sup_{0 \le z \le s}\{z u + g_{pp}(s-z, t)\} + \tilde\e u\Big)\\
&\phantom{xxxxxxx}\bigvee \Big(\max_{0\le \tilde z \le t}\Big\{\tilde z \frac{p(1-u)}{u-p} + g_{pp}(s, t- \tilde z) \Big\} +\tilde\e\frac{p(1-u)}{u-p}\Big).
\end{align*}
The convergence becomes a.s. \ on a subsequence. The upper bound for \eqref{eq:maxpath} now follows by letting $\e \to 0$ and $\tilde\e \to 0$ in the final equation.   \end{proof}

\section{Basic properties of the rate function} 
\label{app:prop}

\begin{proof}[Proof of Theorem \ref{JInt}]
First we prove the existence of limit \eqref{J}. Take $m,n\in\N$ and an error due to the floor function ${\bf x}_{m,n}\in(0,1)^2 $ such that 
$(\fl{(m+n)s},\fl{(m+n)t})=(\fl{ms},\fl{mt})+(\fl{ns},\fl{nt})+{\bf x}_{m,n}$. We have
\begin{align*}
\P\{&G_{\fl{(m+n)s},\fl{(m+n)t}}\geq(m+n)sr\}\\
& \geq\P\{G_{\fl{ms},\fl{mt}}+G_{(\fl{ms},\fl{mt}),(\fl{(m+n)s},\fl{(m+n)t})}\geq(m+n)r\},  \quad &&\text{by superadditivity}\\
& \geq\P\{G_{\fl{ms},\fl{mt}}\geq mr\}\P\{G_{\fl{ns},\fl{nt}}\geq nr\}\P\{G_{\fl{{\bf x}_{m,n}}}\geq 0\}, \quad &&\text{by independence}.
\end{align*}
By \eqref{rval} $ \P\{G_{\fl{x_{m,n}}}\geq 0\}=1$. Take logarithms in the last inequality; then by Fekete's lemma
the limit 
\[
\lim_{N\to\infty}N^{-1}\log\P\{G_{\fl{Ns}, \fl{Nt}}\geq Nr\}
\]
exists for any $(s,t)\in\R^2\setminus\{0\}$ and $r\in[0,s]$ and in fact equals  $\sup_{N} N^{-1} \log \P\{G_{\fl{Ns},\fl{Nt}}\geq Nr\}$. 
The value of the limit is now denoted by $- J_{s,t}(r)$.

From the superadditivity of $G $ we can also obtain the convexity of the limit. Pick any $\lambda\in(0,1)$ and define the triple $((s,t),r)=\lambda((s_1,t_1),r_1)+(1-\lambda)((s_2,t_2),r_2)$ with $r_1\in[0,s_1]$ and $r_2\in[0,s_2]$. Then
\begin{align*}
&N^{-1}\log\P\{G_{\fl{Ns},\fl{Nt}}\geq Nr\}\\
&\hspace{2cm}\geq \lambda(\lambda N)^{-1}\log \P\{G_{\fl{N\lambda s_1},\fl{N\lambda t_1}}\geq N\lambda r_1\}\\
&\hspace{2.5cm}+(1-\lambda)((1-\lambda )N)^{-1}\log\P\{G_{\fl{N(1-\lambda)s_2},\fl{N(1-\lambda)t_2}}\geq N(1-\lambda)r_2\}.
\end{align*}
Multiply both sides by $-1$ and invert the sign of the inequality to obtain for $N\to\infty$ 
\be
J_{s,t}(r)\leq\lambda J_{s_1,t_1}(r_1)+(1-\lambda)J_{s_2,t_2}(r_2).
\ee
From \eqref{rval} we know that $J$ is finite and we have just proven that it is also convex. This implies that $J$ is continuous on $A$ and upper semicontinuous on the whole set $\bar A$, from Theorems $10.1$ and $10.2$ in \cite{Roc-70}. Moreover, $J_{s,t}(r)$ on $A$ can be uniquely extended to a continuous function on $\bar A$ by Theorem 10.3 in \cite{Roc-70}. 

Finally, the law of large numbers for the last passage time implies $J_{(s,t)}(r)=0$ for $r<g_{pp}(s,t)$ and then by continuity for $r\leq g_{pp}(s,t)$. Use the same method of proof of Proposition $3.1($b$)$ of  \cite{Com-Yos-11-} to get the concentration inequality: 
\be
\P\{|G_{\fl{Ns},\fl{Nt}}-\E[G_{\fl{Ns},\fl{Nt}}]|\geq N\e\}\leq 2e^{-c\e^2n}\quad\forall n\in\N.
\ee
This holds for a given $(s,t)\in \R^2_+$, and $\e>0$. Constant $c>0$ will depend on $s, t, \e$.
Since $N^{-1}\E[G_{\fl{Ns},\fl{Nt}}]\to g_{pp}(s,t)$, this implies that $J_{s,t}(r)>0$ for $r>g_{pp}(s,t)$ (without excluding the value $\infty$).
\end{proof}

\begin{lemma}[Continuity in the macrosocpic directions] \label{lem:1}
Let $(s,t)\in\R_{>0}^2$ and $\mathbf u_N=(s_{N},t_{N})\in\Z^2_+$ an increasing sequence such that $N^{-1} \mathbf u_N\to (s,t)$. Then for $r\in [0, s)$
\be
\lim_{N\to\infty}N^{-1}\log\P\{G_{s_{N},t_{N}}\geq Nr\}=-J_{s,t}(r).
\ee
\end{lemma}
\begin{proof}
Since $\mathbf u_N$ and $(\fl{Ns}, \fl{Nt})$ are non-decreasing in $N$, for each $N$ we can find two sequences $\ell_N$ and $m_N$ such that
\[\fl{\ell_N(s,t)}\leq\mathbf u_N\leq\fl{m_N(s,t)}\quad \text{with } N-m_N,N-\ell_N=o(N).\]

%For each $N$ fix a path from $(\fl{\ell_Ns}, \fl{\ell_Nt})$ to $\mathbf u_N$ with $K_N$ steps through the sequence of sites $\{\mathbf x_i\}_{1\leq i\leq K_N}$ 
%and from $\mathbf t_N$ to $\fl{m_N \mathbf y}$ with $K'_N$ through the sequence of sites 
%$\{\mathbf y_i\}_{1\leq i\leq K'_N}$. 
%The assumption implies $K_N$ and $K'_N$ are $o(N)$. 

Then it is immediate that 
\[G_{\fl{\ell_Ns}, \fl{\ell_Nt}} \leq G_{s_N, t_N}\leq G_{\fl{m_Ns}, \fl{m_Nt}},\]
which gives 
\begin{align*}
\P\{G_{\fl{m_Ns}, \fl{m_Nt}}\geq Nr\}
&
%\geq\P\Big\{ G_{s_N,  t_N}
%+\sum_{i=1}^{K_N}\om_{\mathbf x_i}\geq Nr\Big\}%\geq %\P\{G_{s_N, t_N}\geq Nr\}\P\Big\{\sum_{i=1}^{K_N}\om_{\mathbf x_i}\geq0\Big\}\\
\ge \P\{G_{s_N, t_N}\geq Nr\} \ge \P\{G_{\fl{\ell_Ns}, \fl{\ell_Nt}}\geq Nr\}.
\end{align*}
Taking the $\varlimsup$ of both sides and by the continuity of the rate function we have
\begin{align*}
\varlimsup_{N\to\infty}N^{-1}\log \P\{G_{s_N, t_N}&\geq Nr\}\leq \lim_{N\to\infty}N^{-1}\log \P\{G_{\fl{m_Ns}, \fl{m_Nt}}\geq Nr\}\\
&\leq \varlimsup_{N\to\infty}m_N^{-1} (\frac{m_N}{N})\log \P\{G_{\fl{m_Ns}, \fl{m_Nt}}\geq m_Nr - (m_N - N)r\}\\
&\leq \varlimsup_{N\to\infty}m_N^{-1} (\frac{m_N}{N})\log \P\{G_{\fl{m_Ns}, \fl{m_Nt}}\geq m_N(r-\e)\} \\
&\phantom{xxxxxxxxxxxxxxxxxx} \text{ for any $\e>0$ and $N$ large enough}\\
&= -J_{s,t}(r - \e).
\end{align*}
Then let $\e \to 0$ and invoke the continuity of $J$ for the upper bound. Same arguments are valid for the lower bound, using $\varliminf_{N \to \infty}$ .
\end{proof}

From Theorem \ref{JInt} we have that $J_{s,t}(r)$ can be continuously extended to the boundary of the domain $A = \{ (s,t,r):  J_{s,t}(r) < \infty \}$,
\[ \partial A = \{s = 0, t \ge 0, r \le 0 \} \cup \{ t = 0, s \ge r \vee 0 \} \cup \{ s = r, t \ge 0 \}. \]
 It will be convenient to understand the values of the continuation of $J_{s,t}(r)$ on $\partial A$.

 For any $s,t > 0$ and $r \le 0$, $J_{s,t}(r) = 0$. Therefore, we will have that 
 \[
 J_{s,0}(r) = J_{0, t}( r) = 0, \quad r \le 0. 
 \]  
 Now for the $r > 0$ case. Since we want $J_{s,0}(r)$ with $(s \ge r)$ continuous we define 
 $J_{s,h}(r) = \lim_{h \to 0} J_{s, h}(r)$. An approximation using thin rectangles as in \cite{Geo-Sep-13-} gives that 
 \[
 J_{s,0}(r) = s I_{\mathcal B}(r/s) = r \log \frac{r}{sp} +  \big(s - r \big)\log \frac{1 - r/s}{1 - p}.
 \]
Recall that $I_{\mathcal B}$ is the Cram\'er rate function for sums of i.i.d. $\om_i \sim$ Bernoulli$(p)$.
This discussion is summarised in Corollary \ref{thm:JJJ}.

\begin{lemma}[Infimal convolutions]\label{srf}
For each $N$ let $L_N$ and $Z_N$ be two independent random variables. Assume their rate functions
\begin{align}
\lambda(s)&=-\lim_{N\to\infty}N^{-1}\log\P\{L_N\geq Ns\},\\
\phi(s)&=-\lim_{N\to\infty}N^{-1}\log\P\{Z_N\geq Ns\}
\end{align}
exists and 
\begin{enumerate}
\item $\lambda(s)$ is finite in $(-\infty,b)$ with $b\in \bar \R$ and $\lambda(s) = \infty$ when $s > b$. 
\item $\lambda$ is continuous at all points for which is finite and lower semi-continuous on $\R$. 
\item $\phi(s)$ is finite for all $s\in\R$. 
\item $\lambda(a_\lambda)=\phi(a_\phi)=0$ for some $a_\lambda,a_\phi\in\R$. 
\end{enumerate}
Then for $r\in \R$
\be
\begin{aligned}\label{eq:af}
\lim_{N\to\infty}N^{-1}\log\P\{L_N+&Z_N\geq Nr\}\\
&=\begin{cases}
-\inf_{a_\lambda\leq s\leq b\wedge (r-a_\phi)}\{\phi(r-s)+\lambda(s)\},\qquad &r>a_\phi+a_\lambda,\\
0,\qquad&r\leq a_\phi+a_\lambda.
\end{cases}
\end{aligned}
\ee
\end{lemma}
\begin{proof}
First observe that the infimum in \eqref{eq:af} is obtained when $s$ satisfies $a_\lambda\leq s\leq b\wedge (r-a_\phi)$.
 
The lower bound follows from the independence of the two random variables
\[\P\{L_N+Z_N\geq Nr\}\geq \P\{Z_N\geq N(r-s)\}\P\{L_N\geq Ns\}.\]

To upper bound for $r\leq a_\lambda+a_\phi$ is immediate. 

We therefore only discuss the case $r>a_\lambda+a_\phi$. Take a finite partition
$a_\lambda=q_{-1}=q_0<\dots<q_{m-1}=b\wedge (r-a_\phi) < q_m=q_{m+1}$.

Use a union bound and the independence of $L_N, Z_N$ to derive 
\begin{align*}
\P\{&L_N+Z_N\geq Nr\}\leq\P\{L_N+Z_N\geq Nr, L_N< Nq_0\}\\
&\hspace{4cm}+\sum_{i=0}^{m-1}\P\{L_N+Z_N\geq Nr, nq_i\leq L_N\leq Nq_{i+1}\}+P\{L_N\geq  Nq_m\}\\
&\leq P\{Z_N \geq N(r-q_0)\}+\sum_{i=0}^{m-1}\P\{Z_N\geq N(r-q_{i+1})\}\P\{L_N\geq Nq_{i}\}+P\{L_N\geq  Nq_m\}.
\end{align*}
Now take the logarithm on both sides, divide by $N$ and finally take $N\to\infty$ to obtain
\begin{align*}
\varlimsup_{N\to\infty}N^{-1}\log\P\{&L_N+Z_N\geq Nr\}\\
&\leq -\min\Big\{\phi(r-q_0),\min_{0\leq i\leq m-1}\{\phi(r-q_{i+1})+\lambda(q_i)\},\lambda(q_m)\Big\}.
\end{align*}
We may simplify the last inequality as 
\[
\P\{L_N+Z_N\geq Nr\}\leq -\min_{-1\leq i\leq m}\{\phi(r-q_{i+1})+\lambda(q_i)\}
\]
This is because $\lambda(q_0) = 0$. Also, if $b \le r - a_{\phi}$ then $\lambda(q_m) = \infty$ and it can be omitted from the minimum. If $b > r - a_{\phi}$ then 
$\phi(r - q_m) = 0$. The result then follows by the continuity of $\lambda$ on $[a_\lambda, b]$ by arbitrarily refining the partition.  
\end{proof}

\section{The critical point in the proof of  Theorem \ref{thm:lam}} 
\label{sec:app1}

We want to find the solutions to 
\be\label{eq:secu}
\begin{aligned}
0=&u^22s[(1-p)(e^\xi-1)+p(1-e^{-\xi})]-up[(1-p)(s+t)(e^\xi+e^{-\xi}-2)+2s(1-e^{-\xi})]\\
&\phantom{xxxxxx}+(e^{-\xi}-1)p((1-p)(s+t)-s),
\end{aligned}
\ee
when $\xi \ge 0$ and prove that the corresponding solution is the minimizing argument for the function
\[
f(u)=sC_\mathcal B^{(u)}(\xi)-tC_\mathcal G^{(\frac{u-p}{u(1-p)})}(-\xi).
\]
The discriminant $\Delta$ of \eqref{eq:secu} is 
\begin{align*}
\Delta
%&=e^{2\xi}(1-p)^2(s+t)^2p^2+e^{-2\xi}(1-p)^2(s+t)^2p^2+6(1-p)^2(s+t)^2p^2\\
%&-8p(1-p)st-4e^{-\xi}[p^2(1-p)^2(s+t)^2-p(1-p)st]-4e^\xi\{p^2(1-p)^2(s+t)^2\\
%&-p(1-p)st\}\\
&=[(1-p)p(s+t)e^{-2\xi}(e^\xi -1)^2]^2+e^{-\xi}4p(1-p)st(e^\xi - 1)^2 \ge 0.
\end{align*}
Therefore two solutions and are given by
\[
u^*_\pm=\frac{p(1-p)(s+t)(e^\xi+e^{-\xi}-2)+2sp(1-e^{-\xi})\pm\sqrt{\Delta}}{2s(2p-1+(1-p)e^\xi-pe^{-\xi})}.
\]
%To be proper candidates, these $u^*$ have to satisfy two features 
%\begin{enumerate}[(1)]
%\item $u^*\in(p,1]$,
%\item $u^*$ have to be two minimum points.
%\end{enumerate}
We begin by checking if $u^*_\pm\in(p,1]$. It is immediate to check that the  $-\sqrt{\Delta}$ solution 
is not larger than $p$ when $\xi > 0$ and $u^*_- < p$, so we focus on the plus one and $u^*_+$. In that case, the following inequalities are equivalent:
\begin{align*}
&\frac{p(1-p)(s+t)(e^\xi+e^{-\xi}-2)+2sp(1-e^{-\xi})+\sqrt{\Delta}}{2s[(1-p)(e^\xi-1)+p(1-e^{-\xi})]}>p\\
&\frac{p(1-p)(s+t)(e^\xi+e^{-\xi}-2)+2sp(1-e^{-\xi})(1-p)+\sqrt{\Delta}-p2s(1-p)(e^\xi-1)}{2s[(1-p)(e^\xi-1)+p(1-e^{-\xi})]}>0\\
&\frac{p(1-p)(t-s)(e^\xi+e^{-\xi}-2)+\sqrt{\Delta}}{2s[(1-p)(e^\xi-1)+p(1-e^{-\xi})]}>0,
\end{align*}
which is immediately true since the numerator and denominator are always positive for $\xi > 0$. 

The other bound
\begin{align*}
&\frac{p(1-p)(s+t)(e^\xi+e^{-\xi}-2)+2sp(1-e^{-\xi})+\sqrt{\Delta}}{2s[(1-p)(e^\xi-1)+p(1-e^{-\xi})]}\leq1 \Longleftrightarrow\\
&\phantom{xxxxxxxxx}\frac{p(1-p)(s+t)(e^\xi+e^{-\xi}-2)+\sqrt{\Delta}-2s(1-p)(e^\xi-1)}{2s[(1-p)(e^\xi-1)+p(1-e^{-\xi})]}\leq0.
\end{align*}
The denominator is always positive for $\xi>0$, therefore the overall fraction is negative if and only if 
\be\label{upl}
p(1-p)(s+t)(e^\xi+e^{-\xi}-2)+\sqrt{\Delta}-2s(1-p)(e^\xi-1)\leq0.
\ee
We will verify \eqref{upl} when $t < \frac{1-p}{p}s$. When $t$ satisfies this condition it automatically satisfies  
\[
t<\Big(\frac{2}{p(1-e^{-\xi})}-1\Big)s.
\]
When this holds, \eqref{upl} can be equivalently written as 
\[
\sqrt{\Delta} < 2s(1-p)(e^\xi-1) - p(1-p)(s+t)(e^\xi+e^{-\xi}-2).
\]
Both sides of the above inequality are positive, so by squaring both sides we reach the equivalent sequence of inequalities 
% 
%
% 
%If $p(1-p)(s+t)(e^\xi+e^{-\xi}-2)>2s(1-p)(e^\xi-1)$ the numerator is automatically positive and so the all fraction is never less than zero. \nicos{and so what does that imply for critical point?}
%%
%%
%%Thus,  we treat the case $p(1-p)(s+t)(e^\xi+e^{-\xi}-2)<2s(1-p)(e^\xi-1)$ for which it is useful to know the hyperbolic function equality $2e^{\xi}\cosh\xi=e^{2\xi}+1$.
%\begin{align*}
%&p(1-p)(s+t)e^{-\xi}(e^\xi- 1)^2 <2s(1-p)(e^\xi-1)\\
%& p (s+t)(1- e^{-\xi}) < 2s
%\end{align*}
%where in the last inequality we have divide both sides by $1-p$. Substitute $\cosh\xi=(e^\xi+e^{-\xi})/2$ and divide both sides by $(e^\xi-1)$
%\begin{align*}
%&p(s+t)e^{-\xi}(e^\xi-1)^2<2s(e^\xi-1)\\
%&p(s+t)(1-e^{-\xi})<2s\\
%&t<\Big(\frac{2}{p(1-e^{-\xi})}-1\Big)s.
%\end{align*}
%If the above condition is satisfied, we can isolate on one side $\sqrt{\Delta}$ in \eqref{upl} and square both sides
\begin{align*}
%&p(1-p)(s+t)(e^\xi+e^{-\xi}-2)+\sqrt{\Delta}-2s(1-p)(e^\xi-1)\leq0\\
&4(1-p)^2p^2(s+t)^2(\cosh\xi-1)^2+8stp(1-p)(\cosh\xi-1)\\
&\hspace{1cm}\leq4p^2(1-p)^2(s+t)^2(\cosh\xi-1)^2+4s^2(1-p)^2(e^{2\xi}+1-2e^{\xi})\\
&\hspace{6cm}-8sp(1-p)^2(s+t)(\cosh\xi-1)(e^\xi-1)\\
%&8stp(\cosh\xi-1)\leq8s^2(1-p)e^{\xi}(\cosh\xi-1)\\
%&\hspace{6cm}-8sp(1-p)(s+t)(\cosh\xi-1)(e^\xi-1)\\
& \Longleftrightarrow stp(e^\xi-p(e^{\xi}-1))\leq s^2(1-p)e^{\xi}-s^2p(1-p)(e^\xi-1)\\
%&stp(e^\xi(1-p)+p)\leq s^2(1-p)(e^\xi(1-p)+p)\\
&\Longleftrightarrow t\leq \frac{(1-p)}{p}s,
\end{align*}
which is true from our hypothesis. This implies that $u^* \in (p,1]$
It remains to argue that $u^*_+$ is the minimizing point.
%
% when $u^*\in(p,1]$ is satisfied. Since computing the second derivative of the two logarithm generating functions is demanding it is quicker to study the sign of their first derivative since most of the calculus has already done. 

For $\xi>0$, the derivative is positive whenever 
\begin{align*}
%&s\frac{\partial C^{(u)}_\mathcal B(\xi)}{\partial u}-t\frac{\partial C^{(\frac{u-p}{u(1-p)})}_\mathcal G(\xi)}{\partial u}>0\\
&s\frac{e^\xi-1}{1+u(e^\xi-1)}-t\frac{p(p-1)(e^{-\xi}-1)}{u^2(1+p(e^{-\xi}-1))-up[1+e^{-\xi}+p(e^{-\xi}-1)]+p^2e^{-\xi}} = \frac{N(u, \xi)}{D(u, \xi)}>0.
\end{align*}
The numerator $N(u, \xi)$ is given by the right hand side of \eqref{eq:secu} and by what we discussed up to this point, for $\xi > 0$ and $t < \frac{1-p}{p}s$  
\[
N(u,\xi)
\begin{cases}
\geq 0&\qquad \text{if }u\in[u^*_+,1] ,\\
<0&\qquad \text{if }u\in(p,u^*_+). 
\end{cases}
\]
The denominator is 
\[
D(u,\xi)=[1-u+ue^\xi][u^2(1- p(1 -e^{-\xi}))-up[1+e^{-\xi}+p(e^{-\xi}-1)]+p^2e^{-\xi}].
\]
The first factor is always positive for this reason we study the sign of the parabola in the second factor. The coefficient of $u^2$ is positive, for every $\xi >0$ and the factor itself has zeros $u^{**}_\pm$ given by 
%
%\begin{align*}
%u^{**}_\pm&=\frac{p(1+e^{-\xi}+p(e^{-\xi}-1))\pm\sqrt{p^2(1+e^{-\xi}+p(e^{-\xi}-1))^2-4p^2e^{-\xi}(1+p(e^{-\xi}-1))}}{2(1+p(e^{-\xi}-1))}\\
%&=\frac{p[1+e^{-\xi}+p(e^{-\xi}-1)\pm\sqrt{(1+e^{-\xi})^2-4e^{-\xi}+p^2(e^{-\xi}-1)^2-2p(e^{-\xi}-1)^2}]}{2(1+p(e^{-\xi}-1))}\\
%&=\frac{p[1+e^{-\xi}+p(e^{-\xi}-1)\pm(e^{-\xi}-1)(1-p)]}{2(1+p(e^{-\xi}-1))}
%\end{align*}
%from which we obtain 
$u^{**}_-=p$, $u^{**}_+=\frac{pe^{-\xi}}{1-p+pe^{-\xi}}<u^{**}_- $. 
Since our range is $u > p$, the second factor, and hence $D(u, \xi) > 0$. 
Overall, 
\[
\frac{N(u,\xi)}{D(u, \xi)} 
\begin{cases}
\geq 0&\qquad \text{if }u\in[u^*_+,1] ,\\
<0&\qquad \text{if }u\in(p,u^*_+). 
\end{cases}
\]
Therefore $u^*_+$ is a global minimum. \qed

\bibliographystyle{plain}
\end{document}